\documentclass[11pt, a4paper]{article}
\usepackage[hidelinks]{hyperref}
\usepackage[total={16.2cm,25cm}]{geometry}
\usepackage{amsmath}
\usepackage{amssymb}
\usepackage{enumerate}
\usepackage{amsthm}
\usepackage{mathtools}
\usepackage{graphicx}
\usepackage[utf8]{inputenc}
\usepackage{xcolor}

\theoremstyle{plain}
\newtheorem{theorem}{Theorem}
\newtheorem{corollary}[theorem]{Corollary}
\newtheorem{proposition}[theorem]{Proposition}
\newtheorem{lemma}[theorem]{Lemma}

\theoremstyle{definition}
\newtheorem{definition}[theorem]{Definition}

\newtheorem{remark}[theorem]{Remark}

\newcommand{\rp}{\operatorname{rp}}
\newcommand{\lp}{\operatorname{lp}}
\newcommand{\sh}{\stackrel{\mathord{\sharp}}{\leq}}
  
\newcommand{\lsh}{\stackrel{\ell\mathord{\sharp}}{\leq}}
\newcommand{\rsh}{\stackrel{r\mathord{\sharp}}{\leq}}
\newcommand{\ls}{\stackrel{\ell\mathord{*}}{\preceq}}
\newcommand{\rs}{\stackrel{r\mathord{*}}{\preceq}}
\newcommand{\s}{\stackrel{\mathord{*}}{\preceq}}
\newcommand{\lslc}{\stackrel{\ell\mathord{*}\mid\ell \scalebox{0.47}{$\square$}}{\preceq}}
\newcommand{\lsc}{\stackrel{\ell\mathord{*}\mid \scalebox{0.47}{$\square$}}{\preceq}}
\newcommand{\lsrc}{\stackrel{\ell\mathord{*}\mid r\scalebox{0.47}{$\square$}}{\preceq}} 
\newcommand{\rslc}{\stackrel{r\mathord{*}\mid \ell \scalebox{0.47}{$\square$}}{\preceq}}
\newcommand{\rsc}{\stackrel{r\mathord{*}\mid \scalebox{0.47}{$\square$}}{\preceq}} 
\newcommand{\rsrc}{\stackrel{r\mathord{*}\mid r\scalebox{0.47}{$\square$}}{\preceq}}
\newcommand{\slc}{\stackrel{\mathord{*}\mid \ell \scalebox{0.47}{$\square$}}{\preceq}}
\newcommand{\scu}{\stackrel{\mathord{*}\mid \scalebox{0.47}{$\square$}}{\preceq}} 
\newcommand{\src}{\stackrel{\mathord{*}\mid r\scalebox{0.47}{$\square$}}{\preceq}}
\newcommand{\is}{\mathord{*}}
\newcommand{\ils}{\ell\mathord{*}}
\newcommand{\irs}{r\mathord{*}}
\newcommand{\ilslc}{\ell\mathord{*} \mid \ell \scalebox{0.47}{$\square$}}
\newcommand{\ilsrc}{\ell\mathord{*} \mid r \scalebox{0.47}{$\square$}}
\newcommand{\ilsc}{\ell\mathord{*} \mid \scalebox{0.60}{$\square$}}
\newcommand{\irslc}{r\mathord{*} \mid\ell \scalebox{0.47}{$\square$}}
\newcommand{\irsrc}{r\mathord{*} \mid r\scalebox{0.47}{$\square$}}
\newcommand{\irsc}{r\mathord{*} \mid \scalebox{0.60}{$\square$}}
\newcommand{\islc}{\mathord{*} \mid\ell \scalebox{0.47}{$\square$}}
\newcommand{\isrc}{\mathord{*} \mid r\scalebox{0.47}{$\square$}}
\newcommand{\isc}{\mathord{*} \mid \scalebox{0.60}{$\square$}}
\newcommand{\clslc}{\ell\mathord{*} \! \! \mid \! \! \ell \scalebox{0.60}{$\square$}}
\newcommand{\clsrc}{\ell\mathord{*} \! \! \mid \! \! r \scalebox{0.60}{$\square$}}
\newcommand{\clsc}{\ell\mathord{*} \! \! \mid \! \! \scalebox{0.60}{$\square$}}
\newcommand{\crslc}{r\mathord{*}  \!\! \mid \! \! \ell \scalebox{0.6}{$\square$}}
\newcommand{\crsrc}{r\mathord{*} \!\! \mid  \!\! r\scalebox{0.6}{$\square$}}
\newcommand{\crsc}{r\mathord{*} \! \! \mid \! \! \scalebox{0.60}{$\square$}}
\newcommand{\cslc}{\mathord{*} \! \! \mid \! \!\ell \scalebox{0.47}{$\square$}}
\newcommand{\csrc}{\mathord{*} \! \! \mid  \!\! r\scalebox{0.47}{$\square$}}
\newcommand{\cscu}{\mathord{*} \! \! \mid  \!\! \scalebox{0.60}{$\square$}}

\begin{document}

\title{New partial orders on Rickart *-rings}

\author{Cecilia Rossana Cimadamore\footnote{\texttt{crcima@criba.edu.ar}} \and 
 Laura Alicia Rueda \footnote{\texttt{laura.rueda@uns.edu.ar}} \and 
   Melina Vanina Verdecchia\footnote{\texttt{mverdec@uns.edu.ar}}
}

\date{}

\maketitle

\begin{center}
Departamento de Matemática, Universidad Nacional del Sur (UNS), Bahía Blanca, Argentina. 
\end{center}

\begin{abstract}
In this paper, we define and study partial orders on a Rickart *-ring obtained by imposing an additional condition on the star and the one-sided star partial orders. For each such order, we show that the down-set of any element is order-isomorphic to a suitable subset of self-adjoint idempotent elements. As an application, we characterize the elements which are below a given element with respect to each of the new orders. We further prove that the down-set of any element is a lattice whenever the ring is regular. We analyze the existence of supremum and infimum of pairs of elements in a regular Rickart *-ring and provide characterizations of these operations whenever they exist. Finally, we extend the latter results for a nonempty subset of elements for regular Baer *-rings.
\end{abstract}

AMS Classifications:  16W10, 06A06, 16E50, 16U90.

Keywords: Rickart *-ring, star order, one-sided star orders, lattice structure.
\section{Introduction}

An involution $*$ on a ring $\mathcal{A}$ is a unary operation satisfying $(a+b)^*= a^*+b^*, (ab)^*=b^*a^* $ and $(a^*)^*=a$, for all $a, b \in \mathcal{A}$. A ring with an involution is called a *-ring. For $a\in\mathcal{A}$, its right annihilator is $a^{\circ}=\{x\in \mathcal{A}: ax=0\}$ and its left annihilator is $\prescript{\circ}{}a=\{x\in \mathcal{A}: xa=0\}$. A Rickart ring is a ring such that for every $a\in \mathcal{A}$, there exist idempotents $p,q$ such that $\prescript{\circ}{}a= \mathcal{A}(1-p)$ and $a^\circ = (1-q) \mathcal{A}$. Every Rickart ring has a multiplicative identity. A *-ring is called a Rickart *-ring if for every $a \in \mathcal{A}$, there exists a self-adjoint idempotent $p$, that is an element satisfying $p^*= p = p^2$, such that $\prescript{\circ}{}a=\mathcal{A} (1-p) $. This self-adjoint idempotent is unique and is denoted by $\lp(a)$. The analogous property for right annihilators also holds when the ring is a Rickart *-ring. So, for every $a\in \mathcal{A}$, there exists a unique self-adjoint element, denoted by $\rp(a)$, such that $a^\circ = (1-\rp(a)) \mathcal{A}$. For further details on Rickart *-rings, see \cite{Be}. 

Recall that an element $a$ in a ring with identity $\mathcal{A}$ is called regular if there is an element $a^-$ (an inner generalized inverse of $a$) such that $aa^-a=a$. The ring $\mathcal{A}$ is said to be regular if every element of $\mathcal{A}$ is regular. An element $a$ in a *-ring $\mathcal{A}$ is Moore–Penrose invertible if there exists $x\in\mathcal{A}$ such that $axa=a$, $xax=x$, $(ax)^\ast=ax$ and $(xa)^\ast=xa$. Such an element $x$ is unique; it is denoted by $a^\dag$ and is called the Moore–Penrose inverse of $a$. It is known that, in a Rickart *-ring, $a$ is regular if and only if $a^\dag$ exists \cite[Corollary 2.13]{RaDiDj}. Moreover, $\lp(a)= aa^\dag$ and $\rp(a)= a^\dag a$ when $a$ is regular. The sets of regular and Moore–Penrose invertible elements are denoted by $\mathcal{A}^{(1)}$ and $\mathcal{A}^{\dag}$, respectively. Clearly, $\mathcal{A}^{\dag}$ is subset of $\mathcal{A}^{(1)}$.

The star partial order was introduced by Drazin in \cite{Drazin} for semigroups. Since then, it has been extensively studied by numerous authors in the contexts of complex matrices and bounded linear operators on infinite-dimensional Hilbert spaces. Moreover, this order has been extended to rings with involution, and in particular to Rickart *-rings. Several equivalent definitions of the star partial order are available in this setting.  For elements $a, b$ in a Rickart *-ring, \cite{Ci2} and \cite{MaRaDj} proposed the following definition: $a\s b$ if and only if $a^*a=a^*b$ and $aa^*=ba^*$. In a Rickart *-ring, $a^*a=a^*b$ if and only if $a=\lp(a)b$, whereas $aa^*=ba^*$ if and only if $ a= b\rp(a)$ (\cite[Lemma 3.2]{Ci2}, \cite[Lemmas 2.4, 2.5]{MaRaDj}). Therefore, 
\begin{equation}
 a\s b \text{ if and only if } \lp(a)b=a=b\rp(a). \tag{$*$}
\end{equation}
Moreover, if $a, b\in \mathcal{A}^\dag$ then $a\s b$ if and only if $aa^\dag= ba^\dag$ and $a^\dag a= a^\dag b$ (\cite{Drazin}, \cite{RaDj}). 

Baksalary and Mitra introduced the one-sided star partial orders, namely, the left and right star orders, for complex matrices in \cite{BaMi}. These orders were subsequently extended to bounded linear operators on Hilbert spaces (\cite{DeWa}, \cite{DoGuMa}, \cite{Cirulis1}). Several alternative definitions of the one-sided star partial orders appear for *-rings in the literature  (\cite{MaRaDj}, \cite{LePaTh}). In this work, we consider the following definitions given by C\={\i}rulis in \cite{CiarXiv}. Let $\mathcal{A}$ be a Rickart *-ring and let $a, b\in \mathcal{A}$: 
\begin{equation}\label{defrs}
 a\rs b \text{ if and only if } a \rp(b)=a=b\rp(a), \tag{$\irs$}
\end{equation}

\begin{equation}\label{defls}
 a\ls b \text{ if and only if } \lp(a)b=a=\lp(b)a. \tag{$\ils$}
\end{equation}

The relations $\rs$ and $\ls$ are partial orders \cite[Theorem 3.3]{CiarXiv}. The definitions of the one-sided star partial orders in \cite{MaRaDj}, \cite{LePaTh} and \cite{CiarXiv} coincide in the setting of regular Rickart *-rings.

In this paper, we introduce new partial orders on a Rickart *-ring $\mathcal{A}$. These binary relations are defined by combining the star, left star, or right star partial orders with one of the following additional conditions:
\begin{align}
a^2 & =ab, \tag{$\ell $\scalebox{0.47}{$\square$}} \\
 a^2 & =ba, \tag{$r$\scalebox{0.47}{$\square$}}\\
 a^2 & =ab=ba, \tag{\scalebox{0.47}{$\square$}}
\end{align}
for $a, b\in \mathcal{A}$. More precisely, we define on $\mathcal{A}$ the binary relations  $\rslc $, $\rsrc$, $\rsc $, $\lslc$, $\lsrc $, $\lsc $, $\slc$, $\src$ and $\scu$, where, for instance, $a\rslc b$ means that $a\rs b$ and that condition ($\ell $\scalebox{0.47}{$\square$}) is satisfied. The remaining relations are defined analogously. We first show that these binary relations define partial orders on a Rickart *-ring. Furthermore, these relations, along with conditions $(\ell \scalebox{0.47}{$\square$})$, $(r \scalebox{0.47}{$\square$})$, and $(\scalebox{0.47}{$\square$})$ are closely related to several well-known partial orders studied in the literature.

For example, if $a$ and $b$ are group invertible elements, then $a\lsc b$ if and only if $a\ls b$ and $a$ is less than or equal to $b$ with respect to the sharp partial order. Moreover, we show that the partial order $\lsrc$ on the set of core invertible elements coincides with the core partial order, and hence relation $\lsrc$ provides an extension of the core partial order to arbitrary elements of a Rickart *-ring. Similarly, the restriction of $\rslc$ to the set of dual core invertible elements coincides with the dual core partial order. The details of these relationships, as well as their connections to the one-sided sharp partial orders, are presented in Section \ref{sec:newpartialorders}.

After introducing the definitions and establishing the fundamental properties of the new partial orders, we proceed to investigate the order structure induced by each of these relations. 

Recall that a partially ordered set $(R, \leq)$ (or simply a poset) is called a lattice if, for every $a, b\in R$, the least upper bound (supremum) $a\vee b$ and the greatest lower bound (infimum) $a\wedge b$ of the set $\{a, b\}$ exist. A map $\phi\colon R\to S$, where $R$ and $S$ are posets, is said to be order-preserving if $\phi(a)\leq \phi(b)$ whenever $a\leq b$. The posets $R$ and $S$ are said to be order-isomorphic, or simply isomorphic, if there exists a bijection $\phi\colon R \to S$ such that both $\phi$ and $\phi^{-1}$ are order-preserving. In that case, $\phi$ is called an isomorphism, and we write $R \simeq S$. See \cite{Sanka} for more details about lattices.

The set of self-adjoint idempotents of a Rickart *-ring $\mathcal{A}$ is denoted by $E(\mathcal A)$. It is well-known that the binary relation defined by $p\leq q$ if and only if $p=pq$, for $p,q\in E(\mathcal A)$, is a partial order on $E(\mathcal A)$. Since $p$ and $q$ are self-adjoint, the equality $p=pq$ is equivalent to $p=qp$. Furthermore, $E(\mathcal A)$ is a lattice \cite[Proposition 1.15]{Be}. 

In this paper, we investigate the order structure of a Rickart *-ring equipped with each of the newly defined partial orders through appropriate subsets of self-adjoint idempotents. Specifically, for every $b\in \mathcal{A}$, we prove that the corresponding down-set $[0, b]^\alpha=\{a\in \mathcal{A}: a \stackrel{\alpha}{\preceq} b\}$, with $\alpha\in \{ \crslc, \crsrc, \crsc, \clslc, \clsrc, \clsc, \cslc, \csrc, \cscu  \}$, is order-isomorphic to a subset $E_b^{\alpha}\subseteq E(\mathcal{A})$, where the elements of $E_b^{\alpha}$ satisfy additional conditions determined by the partial order $ \stackrel{\alpha}{\preceq}$. These order-isomorphisms make it possible to study the induced order structures more effectively, since the analysis of self-adjoint idempotents is particularly tractable. 

As a first consequence of these order-isomorphisms, we characterize the elements which are below an element $b$ with respect to each of the new partial orders. When $a \stackrel{\alpha}{\preceq}b$, we also prove that the poset $[a,b]^{\alpha}= \{x\in \mathcal{A}: a \stackrel{\alpha}{\preceq} x \stackrel{\alpha}{\preceq} b\}$ is isomorphic to the down-set $[0, b-a]^{\alpha}$ for any $\alpha\in \{\crsc, \clsc, \cscu\}$. 

We then assume that the ring is regular and prove that, for any two elements $a_1, a_2\in [0, b]^\alpha$, both the supremum $a_1\vee a_2$ and the infimum $a_1 \wedge a_2$ exist in $[0, b]^\alpha$. Consequently, $[0, b]^\alpha$ is a lattice with respect to each of the partial orders. Moreover, explicit descriptions of the supremum and infimum are obtained. For a regular Rickart *-ring $\mathcal{A}$, we establish necessary and sufficient conditions for the existence of the supremum and infimum of two elements with respect to any of these partial orders, and we provide corresponding characterizations of these operations whenever they exist. We also prove that in a regular Baer *-ring  \cite{Be}, ordered by any $\stackrel{\alpha}{\preceq} $, there exist the supremum and infimum of a nonempty subset of elements of the ring if the subset is upper bounded, providing equational descriptions of them whenever they exist.

The paper is organized as follows. After a brief section of preliminaries, Section \ref{sec:newpartialorders} is devoted to the introduction of the new partial orders and the study of their fundamental properties. In Section \ref{sec:iso_down-sets}, we establish order-isomorphisms between the corresponding down-sets of an element $b$ and suitable subsets $E_b^\alpha$ of self-adjoint idempotents. Finally, in Section \ref{sec:latticestructureofdownsets}, we consider regular rings and investigate their lattice structure.
 
\section{Preliminaries}

We devote this section to recalling some results on Rickart *-rings, as well as relevant properties of the partial orders $\rs$, $\ls$, and $\s$. For completeness, we include some of the proofs. We conclude this section by presenting the definitions of the partial orders related to those defined in this work.
 
For each $a$ in a Rickart *-ring $\mathcal{A}$, let us consider the self-adjoint idempotents $\rp(a)$ and $\lp(a)$. It is not difficult to see that  
\begin{equation}\label{anuladoraderecha}
a^\circ=(1-\rp(a))\mathcal{A}= (\rp(a))^\circ
 \end{equation}

and 
\begin{equation}\label{anuladoraizquierda}
\prescript{\circ}{}a= \mathcal{A}(1-\lp(a))= \prescript{\circ}{}(\lp(a)).
 \end{equation}

In Lemma \ref{Lemma:propiedadesdellpyrp}, we collect several elementary properties of Rickart *-rings. Specific references for each property are given in the proof.

\begin{lemma}\label{Lemma:propiedadesdellpyrp} Let $\mathcal{A}$ be a Rickart *-ring and $a,b \in \mathcal{A} $. The following properties hold. 
\begin{enumerate}[(a)]
\item\label{prop0} $a\rp(a)=a$ and $\lp(a)a=a$. 
 \item\label{prop:rpa*lpa} $\rp(a^*)=\lp(a)$. 
 \item\label{prop1} $a=a\rp(b)$ if and only if $\rp(a)\rp(b)= \rp(a)$ if and only if $\rp(a)\leq \rp(b)$. 
 \item\label{prop2} $a=\lp(b)a$ if and only if $\lp(a)\leq \lp(b)$. 
 \item\label{prop3} If $a=b\rp(a)$ then $a=\lp(b)a$.
 \item\label{prop4} If $a=\lp(a)b$ then  $a=a\rp(b)$. 
 \item\label{prop:lpaMenorIguallpb} If $a= b\rp(a)$ then $\lp(a)\leq \lp(b)$. 
 \item\label{prop:rpaMenorrpb1} If $a= \lp(a)b$ then $\rp(a)\leq \rp(b)$. 
 
 \item\label{prop:rpapIgualp} If $p\in E(\mathcal{A})$ and $p\leq\rp(a)$ then $\rp(ap)=p$. 
\item\label{propiedadesdellpyrp} If $a^2= ab$ then $ \rp(a) a = \rp(a) b$.
\item\label{prop:dfnestrellaconlp} $a^*a=a^*b$ if and only if $a=\lp(a)b$.
\item\label{prop:dfnestrellaconrp} $aa^*=ba^*$ if and only if $ a= b\rp(a)$.
\end{enumerate} 
\end{lemma}
\begin{proof} Let us prove $(\ref{prop0})$. From $(\rp(a))^2=\rp(a)$ and (\ref{anuladoraderecha}), we have $\rp(a)(\rp(a)-1)=0$  if and only if $ a(\rp(a)-1)=0$  if and only if  $ a\rp(a)=a$. Analogously, $\lp(a)a=a$ using (\ref{anuladoraizquierda}). (See \cite[Proposition 2.4$\rm(b)$]{Ci2} and \cite[Proposition 2.1$(\rm b)$]{CiarXiv}).

The proof of $(\ref{prop:rpa*lpa})$ can be found in \cite[Lemma 2.7]{MaRaDj}. 

$(\ref{prop1})$ is an easy consequence of $(\ref{anuladoraderecha})$. Indeed, $a=a\rp(b)$ if and only if  $a(1-\rp(b))=0$ if and only if $\rp(a)(1-\rp(b))=0$ if and only if $\rp(a)= \rp(a)\rp(b)$ (see (3.6) in \cite{CiarXiv}). Similarly, $(\ref{prop2})$ is an easy consequence of $(\ref{anuladoraizquierda})$ (see (3.5) in \cite{CiarXiv}). 

If $a=b\rp(a)$ then $(1-\lp(b))a=(1-\lp(b))b\rp(a)$. Taking into account $(\ref{prop0})$, $(1-\lp(b))a=0$ and in consequence $a=\lp(b)a$. Thus, (\ref{prop3}) holds.

Property $(\ref{prop4})$ is proved similarly to (\ref{prop3}) but post-multiplicating $a=\lp(a)b$ by $1-\rp(b)$ and using $(\ref{prop0})$. 

$(\ref{prop:lpaMenorIguallpb})$ and $(\ref{prop:rpaMenorrpb1})$ are immediate from $(\ref{prop3})$ and $(\ref{prop2})$, and from  ($\ref{prop4}$) and ($\ref{prop1}$), respectively.

The proof of $(\ref{prop:rpapIgualp})$ can be found in \cite[Proposition 2.4$\rm(g)$]{Ci2}.

If $a^2= ab$ then $ (a-b)\in a^{\circ}=\rp(a)^\circ$. So, $\rp(a)(a-b)=0$ and this implies that $ \rp(a) a = \rp(a) b$. Thus, (\ref{propiedadesdellpyrp}) holds.

The proofs of $(\ref{prop:dfnestrellaconlp})$ and $(\ref{prop:dfnestrellaconrp})$ can be found in \cite[Lemma 3.2]{Ci2} and \cite[Lemmas 2.4 and 2.5]{MaRaDj}. 
\end{proof}

We now present several remarks on the partial orders $\rs$, $\ls$ and $\s$ that are essential for our purpose.

\begin{remark}\label{estrellaifflaterales}
In a Rickart *-ring, $a \s b$ if and only if $a \ls b$ and $a \rs b$. Indeed, the ``if'' condition is trivial from the definitions of $\ls$, $\rs$ and $\s$. 
Conversely, if $a\s b$ then $a\rs b$ by Lemma \ref{Lemma:propiedadesdellpyrp}($\ref{prop4}$). Finally, $a\ls b$ by Lemma \ref{Lemma:propiedadesdellpyrp}($\ref{prop3}$). 

\end{remark}

\begin{remark}\label{phibiendefinida} If $a\rs b$, by Lemma \ref{Lemma:propiedadesdellpyrp}(\ref{prop1}), we have $\rp(a)\leq \rp(b)$. Similarly, by Lemma \ref{Lemma:propiedadesdellpyrp}(\ref{prop2}), if $a\ls b$ then $\lp(a)\leq \lp(b)$. Finally, in view of Remark \ref{estrellaifflaterales} (see also \cite[Corollary 3.4]{Ci2}), if $a\s b$ then $\lp(a)\leq \lp(b)$ and $\rp(a)\leq \rp(b)$.  
\end{remark}

\begin{remark}\label{ls isomorfo rs} It is easy to verify that the posets $\left(\mathcal{A}, \ls\right)$ and $\left(\mathcal{A}, \rs\right)$ are isomorphic via the involution $*$ as a map 
\[*\colon \left(\mathcal{A}, \ls\right) \to \left(\mathcal{A}, \rs\right)\]
 which assigns to each element $a$ its adjoint $a^*$. This map is clearly bijective. Moreover, 
$a\ls b$ if and only if $\lp(a)b=a=\lp(b)a$, which is equivalent to $b^*\lp(a)= a^*= a^*\lp(b)$. By Lemma \ref{Lemma:propiedadesdellpyrp}(\ref{prop:rpa*lpa}) these equalities are in turn equivalent to $b^*\rp(a^*)= a^*= a^* \rp(b^*) $, and hence to $a^*\rs b^*$. Therefore, both $*$ and its inverse (which coincides with $*$) are order-preserving.
\end{remark}

In the remainder of this section, we recall the definitions of some partial orders on rings from the literature that are related to the partial orders introduced in this paper.

An element $a\in\mathcal{A}$ is group invertible if there exists $x\in\mathcal{A}$ such that $axa= a$,  $xax=x$  and  $ax= xa$. This element is also unique and is denoted by $a^\sharp$; it is called the group inverse of $a$ \cite{Ra}. An element $a^\oplus \in \mathcal{A}$ is called the core inverse of $a$ if it satisfies $aa^\oplus a = a$, $a^\oplus \mathcal{A} =a\mathcal{A}$ and $ \mathcal{A}a^\oplus =\mathcal{A}a^*$. An element $a_\oplus\in \mathcal{A}$ is called the dual core inverse of $a$ if it satisfies $aa_\oplus a = a$, $a_\oplus \mathcal{A} =a^*\mathcal{A}$ and $ \mathcal{A}a_\oplus =\mathcal{A}a$. Whenever they exist, these elements are unique \cite{RaDiDj}. The sets of group invertible, core invertible, and dual core invertible elements of $\mathcal{A}$ are denoted by $\mathcal{A}^\sharp$, $\mathcal{A}^\oplus$ and $\mathcal{A}_\oplus$, respectively. Clearly, $\mathcal{A}^\sharp$, $\mathcal{A}^\oplus$ and $\mathcal{A}_\oplus$ are all subsets of $\mathcal{A}^{(1)}$. 

Mitra introduced the sharp partial order on the set of complex matrices admitting a group inverse in \cite{Mi}. This order was subsequently studied in detail by Mitra, Bhimasankaram and Malik in \cite{MiBhMa}, where several fundamental properties were established. Later, Cvetkovi\'{c}-Ili\'{c}, Mosi\'{c} and Wei \cite{CvMoWe}, as well as Jose and Sivakumar \cite{ShSi}, investigated this order in the setting of bounded linear operators on Hilbert spaces. The sharp partial order was subsequently extended to more general settings, including *-rings, in \cite{Mar}, \cite{Ra}, and \cite{RaDj}. Following \cite{RaDj}, for $a,b \in \mathcal{A}^\sharp$, $a \sh b$ if and only if $a^\sharp a=a^\sharp b$ and $aa^\sharp=b a^\sharp $. Moreover, using properties of the group inverse it is not difficult to see that $a^\sharp a=a^\sharp b$ if and only if $a^2=ab$, and, $a a^\sharp =b a^\sharp $ if and only if $ a^2=ba$. Therefore, 
\begin{equation}\label{dfn_sh} 
a \sh b \text{ if and only if } a^2=ab=ba 
\end{equation}

In \cite{MiBhMa}, Mitra, Bhimasankaram and Malik introduced the left and right sharp partial orders on the set of complex matrices admitting a group inverse, by relaxing one of the defining conditions of the sharp partial order. Subsequently, Marovt extended these notions to bounded linear operators on Hilbert spaces in \cite{Ma-One-sidedsharp}. The concepts of left and right sharp partial orders were later generalized to rings in \cite{Ma-One-sidedsharp} and, independently, in \cite{Ra}. Following \cite{Ma-One-sidedsharp}, for $a,b \in \mathcal{A}^\sharp$, the relation $\lsh$ is defined by $a \lsh b$ if and only if $a^\sharp a=a^\sharp b$ and $\prescript{\circ}{}b\subseteq \prescript{\circ}{}a$. Similarly, $\rsh$ is defined by $a \rsh b$ if and only if $aa^\sharp =b a^\sharp $ and $b^\circ\subseteq a^\circ$. Moreover, 
\begin{equation}\label{dfn_lsh}a \lsh b \text{ if and only if } a^2=ab \text{ and } \prescript{\circ}{}b\subseteq \prescript{\circ}{}a, 
\end{equation}
\begin{equation}\label{dfn_rsh}
a\rsh \text{ if and only if } a^2=ba \text{ and  } b^\circ\subseteq a^\circ.
\end{equation}

 The core and dual core partial orders were originally introduced by Baksalary and Trenkler in \cite{BaTr} for complex matrices of index at most one. These notions were later studied in the context of bounded linear operators on Hilbert spaces in \cite{CvMoWe}, and independently in \cite{Mar}. Subsequently, in \cite{Ra}, the core and dual core partial orders were extended to Rickart *-rings through annihilator conditions rather than generalized inverses. Independently, these orders were introduced in \cite{RaDj} on arbitrary *-rings through the core and dual core inverses. More precisely, for $a,b\in\mathcal{A}^\oplus$, the relation $a \leq^\oplus b$ is defined by $a^\oplus a= a^\oplus b$ and $aa^\oplus = ba^\oplus$. Similarly, for $a,b\in\mathcal{A}_\oplus$,  $a \leq_\oplus b$ if and only if $ a_\oplus a= a_\oplus b $ and $ aa_\oplus = ba_\oplus$.

\section{New partial orders}\label{sec:newpartialorders}

 The purpose of this section is to introduce all the new binary relations on a Rickart *-ring proposed in this paper. These relations are obtained by combining the partial orders $\rs$, $\ls$ and $\s$ with one of the condition ($\ell$\scalebox{0.47}{$\square$}), ($r$\scalebox{0.47}{$\square$}) and (\scalebox{0.47}{$\square$}). First, we prove that all these relations define partial orders on a Rickart *-ring. We then study some of their basic properties. The section concludes with an analysis of the relationships between these new partial orders and several well-known partial orders, including the core, dual core, and sharp partial orders.

\begin{definition}\label{dfn:rscuadrados}
Let $\mathcal{A}$ be a Rickart *-ring and $a,b \in \mathcal{A} $. Then 
\begin{equation}
 a \rslc b \text{ if and only if } a \rp(b)=a=b\rp(a) \text{ and } a^2 =ab, \tag{$\irslc$}
\end{equation}

\begin{equation}
 a \rsrc b \text{ if and only if } a \rp(b)=a=b\rp(a) \text{ and } a^2 =ba, \tag{$\irsrc$}
\end{equation}

\begin{equation}
 a \rsc b \text{ if and only if } a \rp(b)=a=b\rp(a) \text{ and } a^2 =ab=ba. \tag{$\irsc$}
\end{equation}
\end{definition}

Clearly, 
\begin{equation}\label{rsc iff rslc rsrc}
a \rsc b \text{ if and only if } a \rslc b \text{ and } a \rsrc b, 
\end{equation}
and 
\begin{equation}\label{rsconcuadradosimplicanrs}
a\stackrel{_{\alpha}}\preceq b \text{ implies }a\rs b, \text{ for each } \alpha\in \{ \crslc,\crsrc, \crsc \}.
\end{equation}

\begin{theorem}\label{teo:rspartialorders}
Let $\mathcal{A}$ be a Rickart *-ring. The binary relations $\rslc$, $\rsrc$ and $\rsc$ are partial orders.
\end{theorem}
\begin{proof} The relations $\rslc$ and $\rsrc$ are reflexive by Lemma \ref{Lemma:propiedadesdellpyrp}(\ref{prop0}). Since $\rs$ is a partial order on $\mathcal {A}$ and according to (\ref{rsconcuadradosimplicanrs}), $\rslc$ and $\rsrc$ are clearly antisymmetric. We prove now that the transitive property holds for these relations. Let $a, b, c \in \mathcal{A}$ be such that $ a \stackrel{_{\alpha}}\preceq b$ and $b \stackrel{_{\alpha}}\preceq c$ with $\alpha\in \{ \crslc, \crsrc \}$. Since $\rs$ satisfies the transitivity property, it is clear that $a\rs c$. Then for  $\rslc$, it only remains to show that condition $ac=a^2$ holds. By Lemma \ref{Lemma:propiedadesdellpyrp}(\ref{propiedadesdellpyrp}), since $b^2=bc$,  $\rp(b)c= \rp(b) b$. Then, from $a= a\rp(b)$ and $ab= a^2$, it follows that $a c= a \rp(b) c= a \rp(b) b= a b= a^2$. To see the transitivity of $\rsrc$, we only need to prove that $ca=a^2$. Indeed, from $a= b\rp(a)$, $cb=b^2$ and $ba=a^2$, we deduce that $ca= c b\rp(a)= b b \rp(a)= ba  = a^2$.

Finally, combining that $\rslc$ and $\rsrc$ are partial orders with the equivalence in (\ref{rsc iff rslc rsrc}), we conclude that $\rsc$ is also a partial order.
\end{proof}

\begin{definition}\label{dfn:lscuadrados}
Let $\mathcal{A}$ be a Rickart *-ring and $a,b \in \mathcal{A} $. Then 
\begin{equation}
 a \lslc b \text{ if and only if }  \lp(a)b=a=\lp(b)a \text{ and } a^2 =ab, \tag{$\ilslc$}
\end{equation}

\begin{equation}
 a \lsrc b \text{ if and only if }  \lp(a)b=a=\lp(b)a\text{ and } a^2 =ba, \tag{$\ilsrc$}
\end{equation}

\begin{equation}
 a \lsc b \text{ if and only if }  \lp(a)b=a=\lp(b)a\text{ and } a^2 =ab=ba. \tag{$\ilsc$}
\end{equation}
\end{definition}

As an immediate consequence of the definitions given above we have that 

\begin{equation}\label{lsc iff lslc lsrc}
 a \lsc b \text{ if and only if } a \lslc b \text{ and } a \lsrc b, 
\end{equation}
and 
\begin{equation}\label{lsconcuadradosimplicanls}
a\stackrel{_{\alpha}}\preceq b \text{ implies }a\ls b, \text{ for each } \alpha\in \{ \clslc,\clsrc, \clsc \}.
\end{equation}

\begin{remark}\label{lsrc lsls isomorfos lslc lsrc} It is easy to see that $a$ and $b$ satisfy $(\ell \scalebox{0.47}{$\square$})$ if and only if $a^*$ and $b^*$ satisfy $(r \scalebox{0.47}{$\square$})$, and, $a$ and $b$ satisfy $(\scalebox{0.47}{$\square$})$ if and only if $a^*$ and $b^*$ satisfy $(\scalebox{0.47}{$\square$})$. Then, by Remark \ref{ls isomorfo rs}, we have that 
\[a \lslc b \text{ if and only if }  a^*\rsrc b^*,\] 
\[ a \lsrc b \text{ if and only if }  a^*\rslc b^*\] and \[a \lsc b \text{ if and only if } a^*\rsc b^*.\]  
\end{remark}

In order to prove that $\lslc$, $\lsrc$, and $\lsc$ are partial orders, we may argue in a way similar to the proof of Theorem \ref{teo:rspartialorders}. However, this fact also follows easily as a consequence of Remark \ref{lsrc lsls isomorfos lslc lsrc} and Theorem \ref{teo:rspartialorders}.

\begin{theorem}\label{teo:lspartialorders}
Let $\mathcal{A}$ be a Rickart *-ring. The binary relations $\lslc$, $\lsrc$ and $\lsc$ are partial orders.
\end{theorem}

\begin{definition}\label{dfn:scuadrados}
Let $\mathcal{A}$ be a Rickart *-ring and $a,b \in \mathcal{A} $. Then

\begin{equation}
 a \slc b \text{ if and only if }  \lp(a)b=a=b\rp(a)\text{ and } a^2 =ab, \tag{$\islc$}
\end{equation}

\begin{equation}
 a \src b \text{ if and only if }  \lp(a)b=a=b\rp(a)\text{ and } a^2 =ba, \tag{$\isrc$}
\end{equation}

\begin{equation}
 a \scu b \text{ if and only if }  \lp(a)b=a=b\rp(a)\text{ and } a^2 =ab=ba. \tag{$\isc$}
\end{equation}

\end{definition}  

Clearly, 

\begin{equation}\label{scu iff slc src}
 a \scu b \text{ if and only if } a \slc b \text{ and } a \src b,\end{equation}
 and 
\begin{equation}\label{sconcuadradosimplicans}
a\stackrel{_{\alpha}}\preceq b \text{ implies }a\s b, \text{ for each } \alpha\in \{ \cslc,\csrc, \cscu \}.
\end{equation}
By Remark \ref{estrellaifflaterales}, the following equivalences are immediate:

\begin{equation}\label{slc iff lslc rslc}
 a \slc b \text{ if and only if } a \lslc b \text{ and } a \rslc b, 
\end{equation}
\begin{equation}\label{src iff lsrc rsrc}
 a \src b \text{ if and only if } a \lsrc b \text{ and } a \rsrc b,
\end{equation}
\begin{equation}\label{scu iff lsc rsc}
a \scu b \text{ if and only if } a \lsc b \text{ and } a \rsc b.
\end{equation}
\begin{theorem}
Let $\mathcal{A}$ be a Rickart *-ring. The binary relations $\slc$, $\src$ and $\scu$ are partial orders.
\end{theorem}
\begin{proof} It is a consequence of (\ref{slc iff lslc rslc}), (\ref{src iff lsrc rsrc}), (\ref{scu iff lsc rsc}), Theorem \ref{teo:rspartialorders} and Theorem \ref{teo:lspartialorders}.  
\end{proof}

Henceforth, whenever the order under consideration is not explicitly specified and we simply write $\stackrel{\alpha}{\preceq}$, it will be understood that $\stackrel{\alpha}{\preceq}$ denotes any of the partial orders introduced in Definitions \ref{dfn:rscuadrados}, \ref{dfn:lscuadrados}, or \ref{dfn:scuadrados}.

\begin{lemma} Let $\mathcal{A}$ be a Rickart *-ring. The following properties are satisfied in the poset $\left(\mathcal{A}, \stackrel{\alpha}{\preceq}\right)$.
 \begin{enumerate}[(a)]
  \item $0$ is the least element in $\left( \mathcal{A}, \stackrel{\alpha}{\preceq}\right)$, for all $\alpha$. 
  
  \item If $p, q\in E(\mathcal{A})$, $p\leq q$ if and only if $p \stackrel{\alpha}{\preceq} q$, for all $\alpha$.
  
  \item $a\stackrel{\alpha}{\preceq} 1$ if and only if $a\in E(\mathcal{A})$, for all $\alpha$.
  
  \item Every left invertible element is maximal when $\alpha \in \{ \crslc, \crsrc, \crsc, \cslc, \csrc, \cscu\}$. 
  
 \item Every right invertible element is maximal when $\alpha \in \{ \clslc, \clsrc, \clsc, \cslc, \csrc, \cscu \}$.  
  
 \end{enumerate}

\end{lemma}

We do not include the proof of the above lemma since it is analogous to the proof of \cite[Lemma 3.1]{Ci2} and \cite[Proposition 4.1]{CiarXiv} for the partial orders $\s$ and $\rs$.

\begin{remark} Maximal elements in $\left(\mathcal{A}, \stackrel{\alpha}{\preceq}\right)$ are not necessarily invertible. Indeed, when the involution is given by the transpose in the Rickart *-ring $M_2(\mathbb{Z}_3)$ of $2 \times 2$ matrices over $\mathbb{Z}_3$, the non-invertible matrix $M= \begin{pmatrix}
0 & 1 \\
0 & 0
\end{pmatrix}$ is a maximal element of $\left( M_2(\mathbb{Z}_3), \stackrel{\alpha}{\preceq}\right)$. We present the proof only for $\alpha= \crslc$, as the arguments for the remaining partial orders are analogous. Let $B=\begin{pmatrix}
x & y \\
z & t
\end{pmatrix} \in M_2(\mathbb{Z}_3) $ be such that $M\rsrc B$. Since $M^2=MB$, we obtain $z=0$ and $t=0$. Furthermore, $MM^T=BM^T$ implies $y=1$. Hence, $B=\begin{pmatrix}
x & 1 \\
0 & 0
\end{pmatrix}$ with $x\in\mathbb{Z}_3 $. If $x=1$ then $\rp (B)= \begin{pmatrix}
2 & 1 \\
1 & 2
\end{pmatrix}$, and thus $M\rp(B)\neq M$. If $x=2$ then $\rp (B)= \begin{pmatrix}
2 & 2 \\
2 & 2
\end{pmatrix}$, so again $M\rp(B)\neq M$. Therefore, $B=M$.   
\end{remark}

Taking into account Remark \ref{lsrc lsls isomorfos lslc lsrc}, the involution $*$ induces isomorphisms$\left(\mathcal{A}, \lslc\right)\simeq \left(\mathcal{A}, \rsrc\right)$, $\left(\mathcal{A}, \lsrc\right)\simeq \left(\mathcal{A}, \rslc\right)$ and $\left(\mathcal{A}, \lsc\right)\simeq \left(\mathcal{A}, \rsc\right)$. These isomorphisms are useful, as they allow properties established for the right star order to be transferred to the left star order, and conversely.

We denote by $[0, b]^\alpha=\{a\in \mathcal{A}: a \stackrel{\alpha}{\preceq} b\}$ the down-set of an element $b$ ordered by $ \stackrel{\alpha}{\preceq}$. The following result will be required in the next section.

\begin{lemma}\label{isomorfismosintervalosbybestrella} Let $\mathcal{A}$ be a Rickart *-ring and $b\in \mathcal{A}$. Then:
\begin{enumerate}[(a)]
 \item\label{isointerle}  $[0, b]^{\ils}$ is isomorphic to $[0, b^*]^{\irs}$.
 \item\label{isointerlerc}  $[0, b]^{\ilsrc}$ is isomorphic to $[0, b^*]^{\irslc}$.
 \item\label{isointerlelc}  $[0, b]^{\ilslc}$ is isomorphic to $[0, b^*]^{\irsrc}$.
 \item\label{isointerlec}  $[0, b]^{\ilsc}$ is isomorphic to $[0, b^*]^{\irsc}$.
\end{enumerate} 
\end{lemma}
\begin{proof} 
Assertion $(a)$ follows directly from Remark \ref{ls isomorfo rs}. The isomorphisms in $(b)$, $(c)$ and $(d)$ are immediate consequences Remark \ref{lsrc lsls isomorfos lslc lsrc}.
 \end{proof}

We conclude this section by relating the new partial orders to some well-known partial orders in the literature. Let $a, b\in\mathcal{A}^\oplus$ be such that $a\leq^\oplus b$. First, observe that $b^\dag $ exists, since $\mathcal{A}^\oplus\subseteq \mathcal{A}^{(1)}$ and the set of regular elements coincides with the set $\mathcal{A}^\dag$ of Moore-Penrose invertible elements in a Rickart *-ring. By \cite[Lemma 2.3]{RaDj} and Lemma \ref{Lemma:propiedadesdellpyrp}(\ref{prop:dfnestrellaconlp}), we have   
\[a^\oplus a= a^\oplus b \text{ if and only if } a^*a= a^*b \text{ if and only if } a=\lp(a)b,  \]
and
\[aa^\oplus = ba^\oplus  \text{ if and only if } a^2= ba .  \]
Since $a= aa^\oplus a= ba^\oplus a$, it follows that $a\mathcal{A}\subseteq b\mathcal{A}$ and it is easy to verify that 
\[a\mathcal A\subseteq b\mathcal A \text{ if and only if } a=bb^\dag a= \lp(b) a.\]

Therefore, if $a, b\in\mathcal{A}^\oplus$ then 
\begin{equation}\label{coreequallsrc}
 a\leq^\oplus b \text{ if and only if } a\lsrc b. 
\end{equation}

Hence, the partial order $\lsrc$ extends the core partial order from core-invertible elements to arbitrary elements of a Rickart *-ring. Analogously, $\rslc$ restricted to the set $\mathcal{A}_{\oplus}$ is exactly the dual core partial order in Rickart *-rings. Indeed, let $a, b\in\mathcal{A}_\oplus$ be such that $a\leq_\oplus b$. Once again, by \cite[Lemma 2.3]{RaDj} and Lemma \ref{Lemma:propiedadesdellpyrp}(\ref{prop:dfnestrellaconrp}), we have
\[a_\oplus a= a_\oplus b \text{ if and only if } aa^*= ba^* \text{ if and only if } a=b\rp(a),  \]
and
\[aa_\oplus = ba_\oplus  \text{ if and only if } a^2= ab .  \]
Moreover, if $a\leq_{\oplus}b$ then $\mathcal{A}a\subseteq \mathcal{A}b$. Taking into account that $b$ is regular, we obtain \[\mathcal Aa\subseteq \mathcal Ab \text{ if and only if } a=ab^\dag b = a\rp(b).\]
So, if $a, b\in\mathcal{A}_\oplus$ then 
\begin{equation}\label{dualcoreequalrslc}
 a\leq_\oplus b \text{ if and only if } a\rslc b.
\end{equation}

Considering the definition of the sharp partial order and $(\ref{dfn_sh})$, it is clear that if $a, b\in\mathcal{A}^\sharp $ 
\begin{equation}\label{eq:lscesshintersectadols}
 a\lsc b \text{ if and only if } a\ls b \text{ and } a\sh b, 
\end{equation}
\begin{equation}\label{eq:rscesshintersectadols}
 a\rsc b \text{ if and only if } a\rs b \text{ and } a\sh b, 
\end{equation}
\begin{equation}\label{eq:scesshintersectadols}
 a\scu b \text{ if and only if } a\s b \text{ and } a\sh b.  
\end{equation}

Let us observe that if $a\ls b$ then $\prescript{\circ}{}b\subseteq \prescript{\circ}{}a$ and $b^\circ\subseteq a^\circ$. Indeed, by $\prescript{\circ}{}b=\prescript{\circ}{}\lp(b)$, if $zb=0$ then $z\lp(b)a=za=0$, and, if $bz=0$ then $\lp(a)bz=0=az$. Analogously, if $a\rs b$ then $\prescript{\circ}{}b\subseteq \prescript{\circ}{}a$ and $b^\circ\subseteq a^\circ$. Thus, considering the definitions of $\lsh$ and $\rsh$ given in (\ref{dfn_lsh}) and (\ref{dfn_rsh}), we can state the following relationships for $a,b\in\mathcal{A}^\sharp $:
\[a\lslc b \text{ if and only if } a\ls b \text{ and } a\lsh b, \]
\[ a\lsrc b \text{ if and only if } a\ls b \text{ and } a\rsh b,\] 
 \[ a\rslc b \text{ if and only if } a\rs b \text{ and } a\lsh b,\] 
 \[ a\rsrc b \text{ if and only if } a\rs b \text{ and } a\rsh b,\]
 \[ a\slc b \text{ if and only if } a\s b \text{ and } a\lsh b,\]and finally,
  \[ a\src b \text{ if and only if } a\s b \text{ and } a\rsh b.\]

\section{Isomorphic representation of down-sets}\label{sec:iso_down-sets}

In this section we establish the isomorphisms between the down-sets of an element $b\in \mathcal{A}$, denoted by $[0, b]^\alpha=\{a\in \mathcal{A}: a \stackrel{\alpha}{\preceq} b\} $ and ordered by each $\stackrel{\alpha}{\preceq}$, and certain poset $E_b^\alpha$, ordered by $\leq$, consisting on self-adjoint idempotents that satisfy certain conditions depending on the partial order $\stackrel{\alpha}{\preceq}$ under consideration. As a consequence, we characterize the elements $a\in \mathcal{A} $ that are below $b$ for each $\stackrel{\alpha}{\preceq}$. We also prove that each segment $[a,b]^{\alpha}$ is isomorphic to the down-set $[0, b-a]^{\alpha}$, when $\alpha\in \{\crsc, \clsc, \cscu\}$. 

We begin by studying the down-set of an element $b$ with respect to the orders $\rslc$, $\rsrc$ and $\rsc$. Our approach relies on a result established by C\={\i}rulis in \cite{CiarXiv} concerning the down-sets $[0, b]^{\irs}=\{a\in \mathcal {A}: a\rs b \}$ and $E^{\irs}_{b}=\{p\in E(\mathcal A ): p\leq \rp(b) \}$. In that work, C\={\i}rulis introduced the maps \[\phi \colon [0, b]^{\irs}\to E^{\irs}_{b}, \text{ defined by }\phi (a)=\rp(a), \]
and 
\[\psi \colon E^{\irs}_{b}\to  [0, b]^{\irs},\text{ defined by } \psi(p)=bp. \]
Observe that, if $a\in [0, b]^{\irs}$, then $\rp(a)\leq \rp(b) $ by Remark \ref{phibiendefinida}. Conversely, if $p\in  E^{\irs}_{b}$, then $p\rp(b)= p$ and $bp\rp(b)= bp= b\rp(bp)$ by Lemma \ref{Lemma:propiedadesdellpyrp}(\ref{prop:rpapIgualp}). Moreover, C\={\i}rulis proved that $\phi$ and $\psi$ are order-preserving and mutually inverse. As a consequence, the following result was obtained.

\begin{theorem}\cite[Theorem 4.2]{CiarXiv} \label{thm:isodecrecienteRestrella} Let $\mathcal{A}$ be a Rickart *-ring and $b\in\mathcal{A}$. Then $[0, b]^{\irs} $ is isomorphic to $E^{\irs}_{b}$. 
\end{theorem}

We first establish several properties that will be needed in the sequel. For each $b\in \mathcal{A}$, we denote the element $\rp(b) b$ by $b_r$. 

\begin{lemma}\label{AuxiliarIso1}
Let $\mathcal{A}$ be a Rickart *-ring and $a,b\in \mathcal{A}$.  
\begin{enumerate}[(a)]
 \item\label{rslc rpbr} If $a \rslc b$ then $\rp(a) b_r \rp(a) =\rp(a) b_r$.
 \item\label{rsrc rpbr} If $a \rsrc b$ then $\rp(a) b_r \rp(a) =b_r\rp(a)$.
\end{enumerate}
 
\end{lemma}
\begin{proof}  
 We begin by proving $(a)$. From  $a^2 = ab$ and $a = b\rp(a)=a\rp(b)$ we obtain $a^2 =(a\rp(b) )(b \rp(a)) = (a\rp(b)) b $. Then, $\rp(b) b \rp(a)-\rp(b) b \in a^\circ=\rp(a)^\circ$, which implies that $\rp(a)\rp(b) b \rp(a)=\rp(a)\rp(b) b$. That is, $\rp(a)b_r \rp(a)=\rp(a)b_r$.

Let us now prove $(b)$. Since  $a^2 = ba$ and $a = b\rp(a)$ then $a^2=(b\rp(a))(b \rp(a)) = b^2\rp(a) $. Thus 
 $\rp(a) b \rp(a)-b\rp(a) \in b^\circ=\rp(b)^\circ$, which implies that  $\rp(b)\rp(a) b \rp(a)=\rp(b)b\rp(a)$. So, $\rp(a)\rp(b) b \rp(a)=\rp(b)b\rp(a)$, that is, $\rp(a) b_r \rp(a) =b_r\rp(a)$. 
\end{proof}

\begin{lemma}\label{AuxiliarIso2}
Let $\mathcal{A}$ be a Rickart *-ring, $b\in \mathcal{A}$, $p,p_1,p_2\in E^{\irs}_{b}$ be such that $p_1 \leq p_2$. 
\begin{enumerate}[(a)]
 \item\label{condicion lc para bp} If $pb_r p=pb_r$ then $(bp)^2 = bpb$. 
 \item\label{condicion rc para bp} If $pb_r p=b_rp$ then $(bp)^2 = b^2p$. 
 \item\label{condicion lc para conp1p2} If $p_1b_r p_1 = p_1b_r$ then $(bp_1)^2 = bp_1bp_2$. 
 \item\label{condicion rc para conp1p2} If $p_1b_r p_1 = b_rp_1$ then $(bp_1)^2 = bp_2bp_1$.
\end{enumerate}
\end{lemma}

\begin{proof}  
 $(a)$ From $p=p\rp(b)$ and $pb_r p=pb_r$, we have $(bp)^2 = bp\rp(b)bp = bp b_r p =bp b_r= bp\rp(b)b=bpb$.

$(b)$  From $p=p\rp(b)$, $pb_r p=b_rp$ and $b\rp(b)=b$, it follows that $(bp)^2 = bpbp = bp\rp(b)bp=bp b_rp= bb_rp=b\rp(b)bp=b^2p$.

$(c)$ From  $p_1 =p_1\rp(b)$, $p_1b_r p_1 = p_1b_r$ and $p_1 = p_1p_2$, we have:
\[
(bp_1)^2=bp_1bp_1 = bp_1\rp(b)bp_1=bp_1b_rp_1p_2=bp_1b_rp_2=bp_1\rp(b)bp_2=bp_1bp_2.
\]

$(d)$ From  $p_i =p_i\rp(b)$, $p_1b_r p_1 = b_rp_1$ and $p_1 = p_2p_1$, we obtain:
\[
(bp_1)^2=bp_1bp_1 = bp_2p_1bp_1=bp_2p_1\rp(b)bp_1=bp_2p_1b_rp_1=bp_2b_rp_1=bp_2\rp(b)bp_1=bp_2bp_1.
\qedhere\]
\end{proof}

\begin{definition}\label{dfn:Ers} Let $\mathcal{A}$ be a Rickart *-ring and $b\in \mathcal{A}$. We define the posets
\begin{align*}
E^{\irslc}_{b} &=\{p\in E(\mathcal A ): p\leq \rp(b) \text{ and } pb_r p=pb_r\},\\ 
E^{\irsrc}_{b} &=\{p\in E(\mathcal A ): p\leq \rp(b) \text{ and } pb_r p=b_r p\},\\
E^{\irsc}_{b} &=\{p\in E(\mathcal A ): p\leq \rp(b) \text{ and } pb_r= b_r p\},
\end{align*}
all of them ordered by $\leq $. 
\end{definition}

We now establish an order-isomorphism between the down-set of an element $b$ and the poset $E_b^{\alpha}$, for each order $\stackrel{_{\alpha}}\preceq$ where $\alpha\in \{\crslc, \crsrc, \crsc\}$. 

\begin{theorem}\label{thm:isodecreRestrellaConCuadrados} Let $\mathcal{A}$ be a Rickart *-ring and $b\in\mathcal{A}$. Then: 
\begin{enumerate}[(a)]
 \item\label{isorslc}  $[0, b]^{\irslc} $ is isomorphic to $E^{\irslc}_{b}$. 
 \item\label{isorsrc}  $[0, b]^{\irsrc} $ is isomorphic to $E^{\irsrc}_{b}$.
 \item\label{isorsc}  $[0, b]^{\irsc} $ is isomorphic to $E^{\irsc}_{b}$.  
\end{enumerate}
 
\end{theorem}
\begin{proof}
In order to prove the isomorphism in $(\ref{isorslc})$, we consider the following restrictions of the maps $\phi$ and $\psi$ defined above: $\overline{\phi}\colon [0, b]^{\irslc} \to E^{\irslc}_{b}$ and 
$\overline{\psi}\colon E^{\irslc}_{b} \to [0, b]^{\irslc}$.
It suffices to show that $\overline{\phi}$ and $\overline{\psi}$ are well-defined, order-preserving and mutually inverse, and hence bijective. By Theorem \ref{thm:isodecrecienteRestrella}, this reduces to verifying that both maps are well-defined and that $\overline{\psi}$ is order-preserving. Let us start by proving the first.
\begin{itemize}
    \item[(i)] If $a \rslc b$ then, by the definition of $\overline{\phi}$, we only need to show that $\rp(a) b_r \rp(a) =\rp(a) b_r$, which follows from Lemma \ref{AuxiliarIso1}$(\ref{rslc rpbr})$.
 
    \item[(ii)] Since $\overline{\psi}$ is the restriction of $\psi$ to the poset $E^{\irslc}_{b}$, what is left is to show that, if $p\in E^{\irslc}_{b}$ then $(bp)^2 = bpb$. This is a direct consequence of Lemma \ref{AuxiliarIso2}$(\ref{condicion lc para bp})$, since by our assumption $p=p\rp(b)$ and $pb_r p=pb_r$.
 
\end{itemize}
To complete the proof of $(\ref{isorslc})$ we now show that $\overline{\psi}$ is order-preserving. Let us consider $p_1,p_2\in E^{\irslc}_{b}$ such that $p_1 \leq p_2$. Thus, we have $bp_1\rs bp_2$ by the definition of $\overline{\psi}$. It remains to prove that $(bp_1)^2 = bp_1bp_2$. But this follows from Lemma \ref{AuxiliarIso2}$(\ref{condicion lc para conp1p2})$, since by hypothesis we have $p_1 =p_1\rp(b)$ and $p_1b_r p_1 = p_1b_r$.

To prove the isomorphism in $(\ref{isorsrc})$ we consider the following restrictions of the maps $\phi$ and $\psi$ defined above:  
$\overline{\phi}\colon [0, b]^{\irsrc} \to E^{\irsrc}_{b}$ and 
$\overline{\psi}\colon E^{\irsrc}_{b} \to [0, b]^{\irsrc}$.

As in $(a)$, we begin by proving that $\overline{\phi}$ and $\overline{\psi}$ are well-defined.
\begin{itemize}
    \item[(i)] If $a \rsrc b$, then the definition of $\overline{\phi}$ reduces the claim to showing that $\rp(a) b_r \rp(a) = b_r \rp(a)$, which is a consequence of Lemma \ref{AuxiliarIso1}$(\ref{rsrc rpbr})$.
     
\item[(ii)] Since $\overline{\psi}$ is the restriction of $\psi$ to the poset $E^{\irsrc}_{b}$, it remains to show that, for $p\in E^{\irsrc}_{b}$, $(bp)^2 = b^2p$. This follows immediately from Lemma \ref{AuxiliarIso2}$(\ref{condicion rc para bp})$, since by our assumption $p=p\rp(b)$ and $pb_r p=b_r p$.

\end{itemize}
To prove that $\overline{\psi}$ is order-preserving, it only remains to show that if $p_1 \leq p_2$ then $(bp_1)^2 = bp_2bp_1$, which follows from Lemma \ref{AuxiliarIso2}$(\ref{condicion rc para conp1p2})$.

 Finally, we prove (\ref{isorsc}). By (\ref{rsc iff rslc rsrc}), we know that $[0, b]^{\irsc}=[0, b]^{\irslc}\cap [0, b]^{\irsrc}$, and it is not difficult to see that $E^{\irsc}_{b}= E^{\irslc}_{b}\cap E^{\irsrc}_{b}$. Then the restrictions maps $\overline{\phi}\colon [0, b]^{\irsc} \to E^{\irsc}_{b}$ and 
$\overline{\psi}\colon E^{\irsc}_{b} \to [0, b]^{\irsc}$ are well-defined and they are clearly mutually inverse isomorphisms by statements $(\ref{isorslc})$ and $(\ref{isorsrc})$.
\end{proof}

As immediate consequences of the isomorphisms established in Theorem \ref{thm:isodecreRestrellaConCuadrados}, we obtain the following characterizations of the elements $a$ lying below a given element $b$ with respect to the orders $\rslc$, $\rsrc$ and $\rsc$.

\begin{corollary}\label{cor:charactbelowrs} Let $\mathcal{A}$ be a Rickart *-ring and $a,b\in \mathcal{A}$.
\begin{enumerate}[(a)]
 \item\label{coritem:charactbelowrslc} $a\rslc b$ if and only if there exists $p\in E(\mathcal{A})$ such that $a= bp$, $p\leq \rp(b)$ and $pb_rp=pb_r$.
 \item $a\rsrc b$ if and only if there exists $p\in E(\mathcal{A})$ such that $a= bp$, $p\leq \rp(b)$ and $pb_rp=b_rp$.
 \item $a\rsc b$ if and only if there exists $p\in E(\mathcal{A})$ such that $a= bp$, $p\leq \rp(b)$ and $pb_r=b_rp$.
\end{enumerate}
Moreover, in each case, the element $p$ is unique.
 \end{corollary}

We show that the partial order $\rslc$, when restricted to the set $\mathcal {A}_{\oplus}$ of dual core invertible elements in Rickart *-rings, coincides with the dual core partial order $\leq_\oplus$ (see (\ref{dualcoreequalrslc})). As a consequence, Corollary \ref{cor:charactbelowrs}(\ref{coritem:charactbelowrslc}) provides a characterization of the elements $a\in\mathcal{A}_{\oplus}$ satisfying $a\leq_\oplus b$ for a given $b\in\mathcal{A}_{\oplus}$.

\begin{corollary}\label{cor:caracterizaciondualcore} Let $\mathcal{A}$ be a Rickart *-ring and $a,b\in \mathcal{A}_{\oplus}$. Then, $a\leq_{\oplus }b$ if and only if there exists (a unique) $p\in E(\mathcal{A})$ such that $a= bp$, $p\leq \rp(b)$ and $pb_rp=pb_r$. 
\end{corollary}

\begin{remark}\label{remark:dualcoreIA}In \cite{Ra}, a generalization of the dual core partial order is introduced on the set $\mathcal{I}_\mathcal{A} = \{a\in \mathcal{A}: a^\circ=p^\circ \text { and } \prescript{\circ}{}a= \prescript{\circ}{}p\text{ for some idempotent } p\in \mathcal{A} \}$ of a Rickart $*$-ring $\mathcal{A}$. Subsequently, \cite[Corollary 4.7]{Marovt2} provides a characterization of the elements lying below a given element $b$ with respect to this order. This characterization is closely related to those obtained in Corollary \ref{cor:charactbelowrs}(\ref{coritem:charactbelowrslc}) and Corollary \ref{cor:caracterizaciondualcore}.
\end{remark}

We now establish analogues to Theorem \ref{thm:isodecreRestrellaConCuadrados} for the partial orders $\lslc$, $\lsrc$ and $\lsc$. In addition, we prove an isomorphism theorem for the order $\ls$, analogous to \cite[Theorem 4.2]{Cirulis1} for $\rs$. To this end, we first introduce the ordered subsets of $E(\mathcal{A})$ that turn out to be order-isomorphic to the down-set of an element $b$ with respect to $\stackrel{_{\alpha}}\preceq$, for each $\alpha\in \{\ils, \clslc, \clsrc, \clsc\}$.  

\begin{definition}\label{dfn:Els} Let $\mathcal{A}$ be a Rickart *-ring and $b\in \mathcal{A}$. We denote $b \lp(b)$ by $b_l$, and we define the posets 
\begin{align*}
      E^{\ils}_{b} &=\{p\in E(\mathcal A ): p\leq \lp(b) \}, \\
 E^{\ilslc}_{b} &=\{p\in E(\mathcal A ): p\leq \lp(b) \text{ and } pb_l p=pb_l\},\\
   E^{\ilsrc}_{b} &=\{p\in E(\mathcal A ): p\leq \lp(b) \text{ and } pb_l p=b_l p\},\\
   E^{\ilsc}_{b} &=\{p\in E(\mathcal A ): p\leq \lp(b) \text{ and } pb_l= b_l p\},
   \end{align*}
   all of them ordered by $\leq$.
\end{definition}

By Lemma \ref{Lemma:propiedadesdellpyrp}$(\ref{prop:rpa*lpa})$, we know that $\lp(b)=\rp(b^*)$ and then $(b_l)^*= \lp(b)b^*= \rp(b^*)b^*=(b^*)_r$. The posets in Definition \ref{dfn:Els} and those in Definition \ref{dfn:Ers}, when considered with respect to the element $b^*$, are related as follows. 

\begin{lemma}\label{igualdadintervalosEbyEbestrella} Let $\mathcal{A}$ be a Rickart *-ring and $b\in \mathcal{A}$. Then:
\begin{enumerate}[(a)]
 \item\label{igualdadle}  $E^{\ils}_{b}=E^{\irs}_{b^*}$.
 \item\label{igualdadlelc}  $E^{\ilslc}_{b}=E^{\irsrc}_{b^*}$.
 \item\label{igualdadlerc}  $E^{\ilsrc}_{b}= E^{\irslc}_{b^*}$.
 \item\label{igualdadlec}  $E^{\ilsc}_{b}=E^{\irsc}_{b^*}$.

\end{enumerate} 
\end{lemma} 
\begin{proof}It is clear that $E^{\ils}_{b}=\{p\in E(\mathcal A ): p\leq \lp(b) \}=\{p\in E(\mathcal A ): p\leq \rp(b^*) \}=E^{\irs}_{b^*}$. Taking into account again that $\lp(b)=\rp(b^*)$, $(b_l)^*= (b^*)_r$ and properties of the involution *, we have that $E^{\ilslc}_{b}=\{p\in E(\mathcal A ): p\leq \lp(b) \text{ and } pb_l p=pb_l\} =\{p\in E(\mathcal A ): p\leq \rp(b^*) \text{ and } p(b^*)_r p=(b^*)_rp\}=E^{\irsrc}_{b^*}$. In the same way, $E^{\ilsrc}_{b}=\{p\in E(\mathcal A ): p\leq \lp(b) \text{ and } pb_l p=b_l p\}= \{p\in E(\mathcal A ): p\leq \rp(b^*) \text{ and } p(b^*)_r p=p(b^*)_r\}= E^{\irslc}_{b^*}$ and finally, $E^{\ilsc}_{b}=\{p\in E(\mathcal A ): p\leq \lp(b) \text{ and } pb_l= b_l p\}=\{p\in E(\mathcal A ): p\leq \rp(b^*) \text{ and } p(b^*)_r =p(b^*)_r\}=E^{\irsc}_{b^*}$.
\end{proof}

In the following theorem we establish the isomorphisms between the down-set of an element $b$ and the corresponding subset of $E(\mathcal A)$ for the orders $\ls$, $\lsrc$, $\lslc$ and $\lsc$. 

\begin{theorem}\label{thm:isodecrecientelestrella} Let $\mathcal{A}$ be a Rickart *-ring and $b$ be an element in $\mathcal{A}$. Then: 
\begin{enumerate}[(a)]
 \item  $[0, b]^{\ils}$ is isomorphic to $E^{\ils}_{b}$.
 \item  $[0, b]^{\ilsrc}$ is isomorphic to $E^{\ilsrc}_{b}$.
 \item  $[0, b]^{\ilslc}$ is isomorphic to $E^{\ilslc}_{b}$.
 \item  $[0, b]^{\ilsc}$ is isomorphic to $E^{\ilsc}_{b}$.
\end{enumerate} 
\end{theorem}
\begin{proof} 
 \begin{enumerate}[(a)]
  \item From Lemma \ref{isomorfismosintervalosbybestrella}$(\ref{isointerle})$, Theorem \ref{thm:isodecrecienteRestrella} and Lemma \ref{igualdadintervalosEbyEbestrella}$(\ref{igualdadle})$, we have that
  \[[0, b]^{\ils} \simeq [0, b^*]^{\irs}\simeq E_{b^*}^{\irs}= E_{b}^{\ils}.\]
  \item In view of Lemma \ref{isomorfismosintervalosbybestrella}$(\ref{isointerlerc})$, Theorem \ref{thm:isodecreRestrellaConCuadrados}$(\ref{isorslc})$ and Lemma \ref{igualdadintervalosEbyEbestrella}$(\ref{igualdadlerc})$, it follows that
  \[[0, b]^{\ilsrc} \simeq [0, b^*]^{\irslc}\simeq E_{b^*}^{\irslc} =  E_{b}^{\ilsrc}.\]
  \item From Lemma \ref{isomorfismosintervalosbybestrella}$(\ref{isointerlelc})$, Theorem \ref{thm:isodecreRestrellaConCuadrados}$(\ref{isorsrc})$ and Lemma \ref{igualdadintervalosEbyEbestrella}$(\ref{igualdadlelc})$ we could deduce that 
  \[[0, b]^{\ilslc} \simeq [0, b^*]^{\irsrc}\simeq E_{b^*}^{\irsrc} =  E_{b}^{\ilslc}.\]
\item Finally, \[[0, b]^{\ilsc} \simeq [0, b^*]^{\irsc}\simeq E_{b^*}^{\irsc} =  E_{b}^{\ilsc}, \]
using Lemma \ref{isomorfismosintervalosbybestrella}$(\ref{isointerlec})$, Theorem \ref{thm:isodecreRestrellaConCuadrados}$(\ref{isorsc})$ and Lemma \ref{igualdadintervalosEbyEbestrella}$(\ref{igualdadlec})$.  \qedhere
 \end{enumerate}
\end{proof}

As an immediate consequence of Theorem \ref{thm:isodecrecientelestrella}, we obtain the following characterization of the elements lying below a given element $b$. 

\begin{corollary}\label{cor:charactbelowlscuadrados} Let $\mathcal{A}$ be a Rickart *-ring and $a,b\in \mathcal{A}$. 
\begin{enumerate}[(a)]
\item $a\ls b$ if and only if there exists $p\in E(\mathcal{A})$ such that $a= pb$, $p\leq \lp(b)$. 
\item\label{coritem:charactbelowlslc} $a\lsrc b$ if and only if there exists $p\in E(\mathcal{A})$ such that $a= pb$, $p\leq \lp(b)$ and $pb_lp=b_lp$.
 \item $a\lslc b$ if and only if there exists $p\in E(\mathcal{A})$ such that $a= pb$, $p\leq \lp(b)$ and $pb_lp=pb_l$.
 \item $a\lsc b$ if and only if there exists $p\in E(\mathcal{A})$ such that $a= pb$, $p\leq \lp(b)$ and $pb_l=b_lp$.
 \end{enumerate}
 Moreover, in each case, the element $p$ is unique.
\end{corollary}

Taking into account Corollary \ref{cor:charactbelowlscuadrados}(\ref{coritem:charactbelowlslc}) and (\ref{coreequallsrc}), we immediately obtain the following result.

\begin{corollary}\label{cor:caracterizacioncore} Let $\mathcal{A}$ be a Rickart *-ring and $a,b\in \mathcal{A}^{\oplus}$. Then,  $a\leq^{\oplus }b$ if and only if there exists (a unique) $p\in E(\mathcal{A})$ such that $a= pb$, $p\leq \lp(b)$ and $pb_lp=b_lp$. 
\end{corollary}

\begin{remark} In \cite{Ra}, a generalization of the core partial order on the set $\mathcal{I}_\mathcal{A}$ was introduced (see Remark \ref{remark:dualcoreIA}). Subsequently, \cite[Theorem 4.5]{Marovt2} provided a characterization of the elements lying below a given element $b$ with respect to this order. This characterization is closely related to the characterizations established in Corollary \ref{cor:charactbelowlscuadrados}(\ref{coritem:charactbelowlslc}) and Corollary \ref{cor:caracterizacioncore}.
\end{remark}

\begin{remark} At the end of Section \ref{sec:newpartialorders}, we relate the  partial orders $\lsc$, $\lslc$, $\lsrc$, $\rslc$ and $\rsrc$ to the sharp and one-sided sharp orders. In \cite[Theorem 13 and Theorem 14]{MaMi} and \cite[Theorem 3.3]{Marovt2}, the elements $a\in \mathcal{A}^\sharp$ lying below $b\in \mathcal{A}^{(1)}$ with respect to the left sharp, right sharp, and sharp partial orders are characterized in the form $a=pb$, where $p$ satisfies certain conditions depending on the partial order under consideration. In each case, one of these conditions involves the element $b_1= \lp(b)b\lp(b)$, which coincides with $b_l$, yielding conditions on $b_1$ analogous to those obtained in Corollary \ref{cor:charactbelowrs} and Corollary \ref{cor:charactbelowlscuadrados}.
\end{remark}

Now we investigate the posets $[0, b]^{\is}$, $[0, b]^{\isrc}$, $[0, b]^{\islc}$ and $[0, b]^{\isc} $, for an arbitrary element $b\in \mathcal {A}$. In \cite{Ci2},  C\={\i}rulis investigated the structure of the $*$-order in Rickart *-rings. In particular, he characterized the elements $a $ lying below a given element $b$ by proving that $ a \s b $ if and only if $ a = b \rp(a) $ and $\rp(a)$ commutes with $b^* b $ \cite[Theorem 3.3]{Ci2}. We first deepen this result by proving that the set of elements below $ b $, ordered by $\s$, is order-isomorphic to the subset of $ E(\mathcal{A})$, ordered by $ \leq$, consisting of those elements that are less than or equal to $\rp(b)$ and commute with $b^* b$. We denote this set by $E^{\is}_{b}$. Next, we study the lower bounds of $b$ with respect to the partial orders $ \src $, $ \slc$ and $\scu$. For each of these partially ordered sets, we prove the existence of order-isomorphisms between the set of elements below $b$ and certain subsets of $ E^{\is}_{b}$ satisfying additional conditions depending on the order under consideration. We begin by introducing the subsets of self-adjoint idempotents that play a central role in this analysis.

\begin{definition}\label{dfn:Es} Let $\mathcal{A}$ be a Rickart *-ring and $b\in \mathcal{A}$. We define the posets
\begin{align*}
E^{\is}_{b} &=\{p\in E(\mathcal A ): p\leq \rp(b) \text{ and } pb^*b= b^*b p \}, \\
E^{\islc}_{b} &=\{p\in E(\mathcal A ): p\leq \rp(b), pb^*b= b^*b p  \text{ and } pb_r p=pb_r\}, \\
E^{\isrc}_{b} &=\{p\in E(\mathcal A ): p\leq \rp(b), pb^*b= b^*b p  \text{ and } pb_r p=b_r p\},\\
E^{\isc}_{b} &=\{p\in E(\mathcal A ): p\leq \rp(b), pb^*b= b^*b p  \text{ and } pb_r= b_r p\}, 
\end{align*}
all of them ordered by $\leq $. 
\end{definition}

Observe that by Lemma \ref{Lemma:propiedadesdellpyrp}$(\ref{prop:dfnestrellaconlp})$ and Lemma \ref{Lemma:propiedadesdellpyrp}$(\ref{prop:dfnestrellaconrp})$, if $ a,b\in \mathcal{A} $ then

\[a\s b \text{ if and only if } a^*a=a^*b \text{ and } aa^*= ba^* .\]

Let us define the map 
\begin{equation}\label{def funcion alpha}
 \alpha\colon [0, b]^{\is}\to E^{\is}_{b}
\end{equation}
 by $\alpha(a)=\rp(a)$. If $a \s b$, by Lemma \ref{Lemma:propiedadesdellpyrp}$(\ref{prop1})$,  we know that $\rp(a) \leq \rp(b)$. In addition, from \cite[Theorem 3.3]{Ci2} it follows that $\rp(a)b^*b= b^*b \rp(a)$ and therefore $\alpha$ is well-defined. Let 
 \begin{equation}\label{def funcion beta}
 \beta\colon E^{\is}_{b}\to  [0, b]^{\is}
 \end{equation}
 be defined by $\beta(p)=bp$. If $p\in E^{\is}_{b}$ then $bp\s b$. Indeed, $(bp)^* (bp)= p b^* b p= pb^* b= (bp)^* b$, since $p$ commutes with $b^*b$. Moreover, it is clear that $(bp)(bp)^*= b(bp)^*$ since $p\in E(\mathcal{A})$. Hence, $\beta $ is well-defined.  

\begin{theorem}\label{thm:isodecrecienteEstrella} Let $\mathcal{A}$ be a Rickart *-ring and $b\in\mathcal{A}$. Then $[0, b]^{\is} $ is isomorphic to $E^{\is}_{b}=\{p\in E(\mathcal A ): p\leq \rp(b) \text{ and } pb^*b= b^*b p \}$. 
\end{theorem}
\begin{proof}  Let us consider the maps $\alpha$ and $\beta$ defined above, and $a \s b$. Since $b\rp(a)= a$ we have $\beta(\alpha(a)) = \beta(\rp(a)) = a$ and, by Lemma \ref{Lemma:propiedadesdellpyrp}$(\ref{prop:rpapIgualp})$, $\alpha(\beta(p)) = \alpha(bp) = p$. Hence $\alpha$ and $\beta$ are mutually inverse, and therefore bijective. It remains to prove that they are order-preserving. Let $a_1 \s a_2$. By Lemma \ref{Lemma:propiedadesdellpyrp}$(\ref{prop1})$, $\rp(a_1) \leq \rp(a_2)$, and thus $\alpha(a_1) \leq \alpha(a_2)$. Finally, to prove that $\beta$ is order-preserving, we have to show that if $p_1 \leq p_2$ then $bp_1\s bp_2$. Indeed,
\[(bp_1)^*(bp_1)=p_1b^*bp_1= p_1b^*b= p_1p_2b^*b=p_1b^*bp_2=(bp_1)^*(bp_2), \] and 
\[(bp_1)(bp_1)^*=bp_1b^*=bp_2p_1b^*= (bp_2)(bp_1)^*.\qedhere\]
 \end{proof}

\begin{theorem}\label{thm:isodecrecienteEstrellalcuadrado} Let $\mathcal{A}$ be a Rickart *-ring and $b\in \mathcal{A}$. 
\begin{enumerate}[(a)]
 \item\label{thmitem:islciso} $[0, b]^{\islc} $ is isomorphic to $E^{\islc}_{b}$. 
 \item \label{thmitem:isrciso} $[0, b]^{\isrc} $ is isomorphic to $E^{\isrc}_{b}$. 
 \item \label{thmitem:isc} $[0, b]^{\isc} $ is isomorphic to $E^{\isc}_{b}$.  
\end{enumerate}
\end{theorem}
\begin{proof}
To prove the isomorphism ($\ref{thmitem:islciso}$), we consider the restrictions $\overline{\alpha}$ and $\overline{\beta}$ of the maps $\alpha$ and $\beta$ defined in $(\ref{def funcion alpha})$ and $(\ref{def funcion beta})$,   
\[\overline{\alpha}\colon [0, b]^{\islc}\to E^{\islc}_{b}\]  and  \[\overline{\beta}\colon E^{\islc}_{b}\to  [0, b]^{\islc}. \] 

 By Theorem \ref{thm:isodecrecienteEstrella}, it remains to prove that $\overline{\alpha}$ and $\overline{\beta}$ are well-defined and that $\overline{\beta}$ is also order-preserving. We begin with the former.
\begin{itemize}
    \item[(i)] If $a\slc b$, then we have $a\rslc b$ by (\ref{slc iff lslc rslc}). From Lemma \ref{AuxiliarIso1}$(\ref{rslc rpbr})$ It follows that $\rp(a)b_r\rp(a)=\rp(a)b_r$. 
\item[(ii)] If $p\in  E^{\islc}_{b}$, by Lemma \ref{AuxiliarIso2}$(\ref{condicion lc para bp})$, we obtain $(bp)^2= bpb$. Since $\overline{\beta} $ is a restriction of $\beta$ to the poset $E^{\islc}_{b}$, we have $bp\in [0, b]^{\islc}$.
\end{itemize}

To establish that $\overline{\beta}$ is order-preserving, it suffices to verify that, for $p_1, p_2\in  E^{\islc}_{b}$, whenever $p_1\leq p_2$, we have $(bp_1)^2= bp_1bp_2$, and this follows from Lemma \ref{AuxiliarIso2}$(\ref{condicion lc para conp1p2})$.

Now, in order to prove the isomorphism in ($\ref{thmitem:isrciso}$), we consider the following restrictions of the maps $\alpha$ and $\beta$ defined above: 
\[\overline{\alpha}\colon [0, b]^{\isrc}\to E^{\isrc}_{b}\]  and  \[\overline{\beta}\colon E^{\isrc}_{b}\to  [0, b]^{\isrc}. \] 
Let us start by proving that $\overline{\alpha}$ and $\overline{\beta}$ are well defined.

\begin{itemize}
    \item[(i)] If $a\src b$, then by (\ref{src iff lsrc rsrc}) we have $a\rsrc b$. From Lemma \ref{AuxiliarIso1}$(\ref{rsrc rpbr})$ we obtain  $\rp(a)b_r\rp(a)=b_r\rp(a)$. 
\item[(ii)] Since $\overline{\beta}$ is the restriction of $\beta$ to the poset $E^{\isrc}_{b}$, it suffices to show that, if $p\in E^{\isrc}_{b}$ then $(bp)^2 = b^2p$, and this follows immediately from Lemma \ref{AuxiliarIso2}$(\ref{condicion rc para bp})$.

\end{itemize}
To conclude the proof of ($\ref{thmitem:isrciso}$), we show that $\overline{\beta}$ is order-preserving. Let $p_1,p_2\in E^{\isrc}_{b}$ with $p_1\leq p_2$. By the definition of $\overline{\beta}$, it follows that $bp_1\s bp_2$. It remains to prove that $(bp_1)^2= bp_2bp_1$, but this follows from Lemma \ref{AuxiliarIso2}$(\ref{condicion rc para conp1p2})$.

Finally, part ($\ref{thmitem:isc}$) follows by an argument analogous to that used in Theorem \ref{thm:isodecreRestrellaConCuadrados}$(\ref{isorsc})$, since $[0, b]^{\isc}=[0, b]^{\islc}\cap [0, b]^{\isrc}$ by (\ref{scu iff slc src}), and it is straightforward to verify that $E^{\isc}_{b}= E^{\islc}_{b}\cap E^{\isrc}_{b}$. 
\end{proof}

We now state a result that follows immediately from the previous theorem and provides a characterization of the elements $a$ in $\mathcal{A}$ lying below a given element $b$ in $\mathcal{A}$ with respect to the orders $\slc$, $\src$ and $\scu$.

\begin{corollary}\label{cor:charactbelows} Let $\mathcal{A}$ be a Rickart *-ring and $a,b\in \mathcal{A}$.
\begin{enumerate}[(a)]
 \item $a\slc b$ if and only if there exists $p\in E(\mathcal{A})$ such that $a= bp$, $p\leq \rp(b)$, $pb^*b= b^*b p$   and $pb_rp=pb_r$.
 \item $a\src b$ if and only if there exists $p\in E(\mathcal{A})$ such that $a= bp$, $p\leq \rp(b)$, $pb^*b= b^*b p$ and $pb_rp=b_rp$.
 \item $a\scu b$ if and only if there exists $p\in E(\mathcal{A})$ such that $a= bp$, $p\leq \rp(b)$, $pb^*b= b^*b p$ and $pb_r=b_rp$.
\end{enumerate}
Moreover, in each case, the element $p$ is unique.
\end{corollary}

We devote the remainder of this section to the study of the interval $[a, b]^\alpha$ for each $\alpha\in \{ \crsc, \clsc, \cscu\} $ .
   
\begin{proposition} Let $\mathcal{A}$ be a Rickart *-ring ordered by $\stackrel{\alpha}{\preceq}$ with $\alpha\in \{ \crsc, \clsc, \cscu\} $  and $a,b\in \mathcal{A}$ such that $a\stackrel{\alpha}{\preceq} b$. Then $[a, b]^{\alpha}=\{x\in \mathcal{A}: a \stackrel{\alpha}{\preceq} x \stackrel{\alpha}{\preceq} b\}$ is isomorphic to $[0, b-a]^{\alpha}$.
\end{proposition}
\begin{proof} First, assume that $\alpha=\crsc$ and $a\rsc b$. We prove that $b-a\rsc b$. Since $a\rsc b$, in particular, we have $aa^*= ba^*$, and hence $ab^*= (ba^*)^*= (aa^*)^*= aa^*=ba^*$. Therefore, $(b-a)(b-a)^*=bb^*-ab^*-ba^*+a a^*= bb^*-ab^*= bb^*-ba^*=b(b-a)^*$. Moreover, $(b-a)\rp(b)= b\rp(b)-a\rp(b)=b-a$, which implies that $b-a\rs b$. Since $a^2=ab=ba$, it follows that $(b-a)^2= (b-a)b= b(b-a)$. Thus, $b-a\rsc b$. Therefore, we conclude that $\rp(b-a)\in E_b^{\irsc}$. 

Clearly, $\rp(b)-\rp(a)$ is self-adjoint. By $\rp(a)\leq \rp(b)$, we have $(\rp(b)-\rp(a))^2=\rp(b)-\rp(a)$, and therefore $\rp(b)-\rp(a)\in E(\mathcal{A})$. Next we prove that $\rp(b-a)= \rp(b)-\rp(a)$. Since $( \rp(b)-\rp(a))\rp(b)=  \rp(b)-\rp(a)$, Lemma \ref{Lemma:propiedadesdellpyrp}($\ref{prop:rpapIgualp}$) yields $\rp(b( \rp(b)-\rp(a)))= \rp(b)-\rp(a)$. Moreover, as $a\rs b$, Lemma \ref{Lemma:propiedadesdellpyrp}($\ref{prop0}$) implies that $b( \rp(b)-\rp(a))=b\rp(b)-b\rp(a)=b-a$. Hence, $\rp(b-a)= \rp(b)-\rp(a)$. 

By Theorem \ref{thm:isodecreRestrellaConCuadrados}$(\ref{isorsc})$, we know that the posets $[a,b]^{\irsc}$ and $[0, b-a]^{\irsc}$ are isomorphic to the posets $[\rp(a), \rp(b)]=\{p\in E_b^{\irsc}: \rp(a)\leq p\leq \rp(b)\}$ and $[0, \rp(b-a)]=\{p\in E_b^{\irsc}: p\leq \rp(b-a)\}$, respectively. We now prove that 
$[\rp(a), \rp(b)] \simeq [0, \rp(b-a)]$. To do that, we consider the maps $f\colon [\rp(a), \rp(b)] \to [0, \rp(b-a)]$ defined by $f(p)=p-\rp(a)$ and $g\colon [0, \rp(b-a)] \to [\rp(a), \rp(b)]$, defined by $g(r)=r+\rp(a)$. Let us observe that if $p\in [\rp(a), \rp(b)]$ then $f(p)\rp(b-a)= (p-\rp(a))(\rp(b)-\rp(a))=f(p)$, that is $f(p)\leq \rp(b-a)$. Clearly, $p-\rp(a)\in E(\mathcal{A})$ and $(p-\rp(a))b_r=b_r(p-\rp(a))$. So, $f$ is well-defined. On the other hand, if $p\leq \rp(b-a)=\rp(b)-\rp(a)$ then $p=p(\rp(b)-\rp(a))$ and so $p\rp(b)=p+p\rp(a)$. But $p\rp(b)=p$ since $p\leq \rp(b)-\rp(a)\leq \rp(b)$. Thus, $p=p+p\rp(a)$ and consequently $p\rp(a)=0$. From this, $(p+\rp(a))\rp(a)=\rp(a)$, $(p+\rp(a))\rp(b)=p+\rp(a)$ and $p+\rp(a)\in E_b^{\irsc}$. Therefore, $g$ is well-defined. It is easy to see that $f$ and $g$ are order-preserving, and they are clearly mutually inverse. Hence, $f$ and $g$ are isomorphisms. 

Using the isomorphism induced by $*$ we have that 
\[[a,b]^{\ilsc}\simeq [a^*,b^*]^{\irsc}\simeq [0, b^*-a^*]^{\irsc}\simeq [0, b-a]^{\ilsc}.\]

Finally, we consider $\alpha=\cscu$ and $a\scu b$. By $a\scu b$, we have $a^* a= a^*b$ and then $a^*a=(a^* a)^*= (a^*b)^*=b^*a$. Thus, $(b-a)^*(b-a)=b^*b-a^*b-b^*a+a^*a= b^*b-a^*b=(b-a)^*b$. Analogously to the previous case, $(b-a)(b-a)^*= b(b-a)^*$ and $(b-a)^2=(b-a)b= b(b-a)$. So, $b-a\scu b$. Let $[\rp(a), \rp(b)]$ be the set $\{p\in E_b^{\isc}: \rp(a)\leq p\leq \rp(b)\}$ and $[0, \rp(b-a)]$ be the set $\{p\in E_b^{\isc}: p\leq \rp(b-a)\}$. Once again, we consider $f\colon [\rp(a), \rp(b)] \to [0, \rp(b-a)]$ defined by $f(p)=p-\rp(a)$ and $g\colon [0, \rp(b-a)] \to [\rp(a), \rp(b)]$, defined by $g(r)=r+\rp(a)$. Clearly, $(p-\rp(a))b^*b=b^*b(p-\rp(a))$ and $(p+\rp(a))b^*b=b^*b(p+\rp(a))$. In view of the proof for $\alpha=\irsc$, we obtain the result for $\alpha=\isc$.
\end{proof}	

Note that, by the above proof, if $a\stackrel{\alpha}{\preceq} b$ with $\alpha\in \{ \crsc, \clsc,\cscu\} $ then the sets $E_{b-a}^{\alpha}$ and $[0, \rp(b-a)]\subseteq E_b^{\alpha}$ coincide.

\section{Lattice structure of down-sets in regular Rickart *-rings}\label{sec:latticestructureofdownsets}

 In this section we prove that the down-set of an element $b$ in a regular Rickart *-ring, ordered by any of the partial orders $\stackrel{\alpha}{\preceq}$, is a lattice. More precisely, whenever $a_1, a_2\in [0, b]^\alpha$, the pair $\{a_1, a_2\}$ admits both a least upper bound and a greatest lower bound in $[0, b]^\alpha$. This result is obtained as an application of the order-isomorphisms established in Section \ref{sec:iso_down-sets}. Later, we study the existence of the supremum and infimum of pairs of elements in a regular Rickart *-ring. A poset is said to have the \emph{upper bound property} if every pair of its elements bounded from above has the least upper bound. We prove that a regular Rickart *-ring ordered by any $\stackrel{\alpha}{\preceq}$ has the upper bound property. We provide characterizations of the supremum and infimum whenever they exist. Finally, we generalize the results for a nonempty subset of elements of the ring for regular Baer *-rings.

It is known that the subset $E(\mathcal A)$ of self-adjoint idempotents of a Rickart *-ring $\mathcal{A}$ is a lattice with $p\vee q= q+\rp(p(1-q))$ and $p\wedge q= p-\lp(p(1-q))$ \cite[Proposition 1.15]{Be}. Moreover, if $\mathcal{A}$ is regular then $p\vee q= q+(p(1-q))^\dag p(1-q) $ and $ p\wedge q=p-p(1-q)(p(1-q))^\dag $. 

Let us recall that a nonempty subset $L'$ of a lattice $L$ is called a sublattice of $L$ if, for every $a, b \in L'$, both $a \vee b$ and $a \wedge b$ are in $L'$, where $\vee$ and $\wedge$ denote the lattice operations of $L$. For any $b\in \mathcal{A}$, the subsets $E_b^{\irs}$ and $E^{\ils}_b$ are sublattices of $E(\mathcal A)$. Indeed, if $p, q\in E_b^{\irs}$ then $p, q\leq \rp(b)$ and so $p\vee q\leq \rp(b)$, since $p\vee q$ is the least upper bound of $p$ and $q$ in $E(\mathcal{A})$. Clearly, $p\wedge q\leq \rp(b)$. Thus, $E_b^{\irs}$ is a sublattice of $E(\mathcal A)$. Similarly, it can be proved that $E_b^{\ils}$ is also a sublattice of $E(\mathcal A)$. 

We begin by establishing several properties that will be useful later. 

\begin{lemma}\label{Lemma:supremos} Let $\mathcal{A}$ be a regular Rickart *-ring, $x\in\mathcal{A}$ and $p, q\in E(\mathcal{A})$. 
\begin{enumerate}[(a)]
 \item\label{Lemma:supremo_rcuadrado} If $pxp=xp$ and $qxq=xq$, then $(p\vee q)x(p\vee q)=x(p\vee q) $. 
\item\label{Lemma:supremolc} If $pxp=px$ and $qxq=qx$, then $(p\vee q)x(p\vee q)=(p\vee q)x $. \end{enumerate} 
\end{lemma}
\begin{proof}Since $p, q\in E(\mathcal{A})$, we know that $p\vee q$ exists and using the expression of $p\vee q$, we have that $(p\vee q)x(p\vee q)= (p\vee q)x( q+ (p(1-q))^\dag p(1-q))  = (p\vee q) x q  + (p\vee q)x (p(1-q))^\dag p(1-q)$. From $qxq= x q$ and $q\leq p\vee q$, we obtain $(p\vee q) x q= (p\vee q) q x q= qxq= x q$. Observe that $(p(1-q))^\dag p(1-q)=\left((p(1-q))^\dag p(1-q) \right)^*= (1-q)p((1-q)p)^\dag =(p-qp)((1-q)p)^\dag $. Combining this with the assumptions on $p$ and $q$, it follows that $(p\vee q)x ((p(1-q))^\dag p(1-q)) = (p\vee q)x (p-qp)((1-q)p)^\dag  =   
  \left( (p\vee q)px p- (p\vee q)qx qp\right) ((1-q)p)^\dag=\left( px p- qx qp\right) ((1-q)p)^\dag = 
  \left( x p- x qp\right) ((1-q)p)^\dag= x (1- q)p ((1-q)p)^\dag= x(p(1-q))^\dag p(1-q)= x\rp(p(1-p))$. Therefore, $(p\vee q)x(p\vee q)=x q+ x\rp(p(1-p))= x(p\vee q)$. So, $(\ref{Lemma:supremo_rcuadrado})$ holds.

Using again the expression of $p\vee q$ we have that $(p\vee q)x(p\vee q)= ( q+ (p(1-q))^\dag p(1-q)) x (p\vee q)  = q x (p\vee q) + (p(1-q))^\dag p(1-q) x(p\vee q)= q x q (p\vee q) + (p(1-q))^\dag p x p (p\vee q)- (p(1-q))^\dag p q xq (p\vee q). $ Thus, by the assumptions, we have that $(p\vee q)x(p\vee q) = q x q + (p(1-q))^\dag p x p - (p(1-q))^\dag p q xq= q x  + (p(1-q))^\dag p x  - (p(1-q))^\dag p q x =(q + (p(1-q))^\dag (p-pq) )x = (p\vee q)x.$ Therefore, $(\ref{Lemma:supremolc})$ is satisfied. 
\end{proof}

\begin{remark}\label{condicionescomplemento} If $p$ is self-adjoint idempotent such that $pxp=xp$, for some $x$, it is straightforward to verify that $(1-p)x(1-p)= (1-p)x$. Similarly, if $pxp=px$, then $(1-p)x(1-p)= x(1-p)$. 
\end{remark}

\begin{corollary}\label{cor:infimos} Let $\mathcal{A}$ be a regular Rickart *-ring, $x\in\mathcal{A}$ and $p,q\in E(\mathcal{A})$.
\begin{enumerate}[(a)]
\item \label{cor:infimorc} If $pxp=xp$ and $qxq=xq$, then $(p\wedge q)x(p\wedge q)=x(p\wedge q) $. 
\item \label{cor:infimolc} If $pxp=px$ and $qxq=qx$, then $(p\wedge q)x(p\wedge q)=(p\wedge q)x $. 
\end{enumerate}

\end{corollary}
\begin{proof} ($\ref{cor:infimorc}$) follows from simple computations together with $p\wedge q= 1-((1-p)\vee (1-q))$, Remark \ref{condicionescomplemento} and Lemma \ref{Lemma:supremos}$(\ref{Lemma:supremolc})$; ($\ref{cor:infimolc}$) is a consequence of Remark \ref{condicionescomplemento} and Lemma \ref{Lemma:supremos}($\ref{Lemma:supremo_rcuadrado}$).  
\end{proof}

We now consider the down-sets of $b$ in $\mathcal{A}$ with respect to the orders $\rslc$, $\rsrc$ and $\rsc$. 

\begin{theorem}\label{thm:latticesrs} Let $\mathcal{A}$ be a regular Rickart *-ring and $b\in \mathcal{A}$. Then $[0, b]^{\irslc} $, $[0, b]^{\irsrc} $ and $[0, b]^{\irsc }$ are sublattices of $[0, b]^{\irs}$. Moreover, $[0, b]^{\irsc }$ is a sublattice of both $[0, b]^{\irslc}$ and $[0, b]^{\irsrc}$.   
\end{theorem}
\begin{proof} Let $p, q\in E_b^{\irs}$. We know that $p\vee q , p\wedge q \in E_b^{\irs}$. If $p, q\in E_b^{\irslc}$, then $p\vee q\in  E_b^{\irslc}$ and $p\wedge q\in  E_b^{\irslc}$ by Lemma \ref{Lemma:supremos}$(\ref{Lemma:supremolc})$ and Corollary \ref{cor:infimos}$(\ref{cor:infimolc})$, respectively. Hence, $E_b^{\irslc}$ is a sublattice of $E_b^{\irs}$. Therefore, by Theorem \ref{thm:isodecrecienteRestrella} and Theorem \ref{thm:isodecreRestrellaConCuadrados}$(\ref{isorslc})$, $ [0,b]^{\irslc}$ is a lattice and, moreover, a sublattice of $[0, b]^{\irs}$.

Analogously, by Lemma \ref{Lemma:supremos}($\ref{Lemma:supremo_rcuadrado}$) and Corollary \ref{cor:infimos}($\ref{cor:infimorc}$), if $p,q \in E_b^{\irsrc}$, then $p\vee q \in E_b^{\irsrc}$ and $p\wedge q \in E_b^{\irsrc}$. It follows that $E_b^{\irsrc}$ is a lattice and, moreover, a sublattice of $E_b^{\irs}$. Finally, Theorem \ref{thm:isodecrecienteRestrella} and Theorem \ref{thm:isodecreRestrellaConCuadrados}$(\ref{isorsrc})$ imply that $ [0,b]^{\irsrc}$ is a lattice that is a sublattice of $[0, b]^{\irs}$. 

The last assertion follows from the equally $E_b^{\irslc}\cap E_b^{\irsrc}= E_b^{\irsc}$ and the isomorphism given in Theorem \ref{thm:isodecreRestrellaConCuadrados}$(\ref{isorsc})$.
\end{proof}

The following result concerns the down-sets of $b$ with respect to the orders $\ls$, $\lslc$, $\lsrc$ and $\lsc$. 
\begin{theorem}\label{thm:latticesls} Let $\mathcal{A}$ be a regular Rickart *-ring and $b\in \mathcal{A}$. Then $[0, b]^{\ilslc} $, $[0, b]^{\ilsrc} $ and $[0, b]^{\ilsc }$ are lattices which, in addition, are sublattices of $[0, b]^{\ils}$. Furthermore, $[0, b]^{\ilsc }$ is a sublattice of both $[0, b]^{\ilslc}$ and $[0, b]^{\ilsrc}$.   
\end{theorem}
\begin{proof} Let $p, q\in E_b^{\ils}$. We know that $p\vee q$ and $p\wedge q$ exist in $E_b^{\ils}$. If $p, q\in E_b^{\ilslc}$, then $p\vee q\in  E_b^{\ilslc}$, by Lemma \ref{Lemma:supremos}$(\ref{Lemma:supremolc})$, and $p\wedge q\in  E_b^{\ilslc}$, by Corollary \ref{cor:infimos}$(\ref{cor:infimolc})$. Thus, $E_b^{\ilslc}$ is a sublattice of $E_b^{\ils}$. Similarly, if $p, q \in E_b^{\ilsrc}$, then $p\vee q \in E_b^{\ilsrc}$, by Lemma \ref{Lemma:supremos}($\ref{Lemma:supremo_rcuadrado}$), and $p\wedge q \in E_b^{\ilsrc}$, by Corollary \ref{cor:infimos}($\ref{cor:infimorc}$). Consequently, $E_b^{\ilsrc}$ is a sublattice of $E_b^{\ils}$. Therefore, by the isomorphisms given in Theorem \ref{thm:isodecrecientelestrella}, we obtain that $ [0,b]^{\ilslc}$ and $ [0,b]^{\ilsrc}$ are lattices and, moreover, sublattices of $[0, b]^{\ils}$. 

Taking into account that $E_b^{\ilslc}$ and $E_b^{\ilsrc}$ are lattices and the equality $E_b^{\ilslc}\cap E_b^{\ilsrc}= E_b^{\ilsc}$, we have that $E_b^{\ilsc}$ is a lattice that is also a sublattice of the posets $E_b^{\ilslc}$, $E_b^{\ilsrc}$ and $E_b^{\ils}$. Finally, the proof is completed using the isomorphisms given in Theorem \ref{thm:isodecrecientelestrella}. 
\end{proof}

Finally, we analyze the cases in which the down-set of $b$ is ordered with respect to the partial orders $\scu$, $\slc$ or $\src$.

\begin{corollary}\label{cor:sublatticesstar} Let $\mathcal{A}$ be a regular Rickart *-ring and $b\in \mathcal{A}$. Then $[0, b]^{\islc} $, $[0, b]^{\isrc} $ and $[0, b]^{\isc }$ are lattices. Moreover, $[0, b]^{\isc }$ is a sublattice of both $[0, b]^{\islc}$ and $[0, b]^{\isrc}$.   
\end{corollary}
\begin{proof} It follows from Theorem \ref{thm:latticesrs}, Theorem \ref{thm:latticesls},  taking into account the equalities $[0, b]^{\isc}=[0, b]^{\ilsc}\cap [0, b]^{\irsc}$, $ [0, b]^{\islc}=[0, b]^{\ilslc}\cap [0, b]^{\irslc} $, and $[0, b]^{\isrc}=[0, b]^{\ilsrc}\cap [0, b]^{\irsrc}$. The last assertion follows from the equality $[0, b]^{\isc }= [0, b]^{\islc}\cap [0, b]^{\isrc}$.    
\end{proof}

Let us observe that, by definition of $\s$, we have that $a\s 1$ if and only if $a\in  E(\mathcal{A})$. In fact, $a^*a=a=aa^*$ and $aa^*=a^*=a^*a$, which imply that $a=a^*=a^2$ (see also \cite[Lemma 3.1]{Ci2}) . Moreover, if $\alpha\in \{\cscu, \cslc, \csrc\}$, then $[0,1]^\alpha\simeq E(\mathcal{A})$.

\begin{remark}\label{remarksupeinfenintervalos} By Theorem \ref{thm:latticesrs}, Theorem \ref{thm:latticesls} and Corollary \ref{cor:sublatticesstar}, if $\mathcal {A}$ is a regular Rickart *-ring and $a_1, a_2\stackrel{\alpha}{\preceq}b$, then both the supremum $a_1\vee a_2$ and the infimum $a_1\wedge a_2$ exist in the down-set $[0,b]^\alpha$. For any $\alpha\in \{\crslc, \crsrc, \crsc, \cslc, \csrc,  \cscu \}$, we know that $a_i= b\rp(a_i)$ by Theorem \ref{thm:isodecreRestrellaConCuadrados} and Theorem \ref{thm:isodecrecienteEstrellalcuadrado}, and so
\[a_1\vee a_2=b(\rp(a_1)\vee \rp(a_2))=b(a_1^\dag a_1\vee a_2^\dag a_2) \]
and
\[a_1\wedge a_2=b(\rp(a_1)\wedge \rp(a_2))= b(a_1^\dag a_1\wedge a_2^\dag a_2).  \] 
While for any $\alpha\in \{\ils, \clslc, \clsrc, \clsc\}$, we have that $a_i=\lp(a_i)b$ by Theorem \ref{thm:isodecrecientelestrella} and then 
\[a_1\vee a_2=(\lp(a_1)\vee \lp(a_2))b=(a_1a_1^\dag \vee a_2a_2^\dag )b \]
and
\[a_1\wedge a_2=(\lp(a_1)\wedge \lp(a_2))b= (a_1a_1^\dag \wedge a_2a_2^\dag )b.  \]

\end{remark}

We now study the existence of supremum and infimum of two elements in a regular Rickart *-ring ordered by each of the partial orders $\stackrel{\alpha}{\preceq}$ defined in this paper. C\={\i}rulis proved that a Rickart *-ring ordered by either $\s$ or $\rs$ has the upper bound property in \cite[Theorem 4.2]{Ci2} and \cite[Theorem 4.6]{CiarXiv}, respectively. Moreover, he proved that if $a_1, a_2\stackrel{\alpha}{\preceq} b$ then $a_1\vee a_2=b(\rp(a_1)\vee \rp(a_2))$ and $a_1\wedge a_2=b(\rp(a_1)\wedge \rp(a_2))$ where $\alpha\in \{\irs,\is \}$. Let us observe that the upper bound property also holds if the Rickart *-ring $\mathcal{A}$ is ordered by $\ls$, since $\left(\mathcal{A}, \ls\right)$ is order-isomorphic to $\left(\mathcal{A}, \rs\right)$ (see Remark \ref{ls isomorfo rs}). But in this case, if $b$ is an upper bound for $a_1$ and $a_2$ then the supremum and the infimum of $\{a_1, a_2\}$ in $\mathcal{A}$ are given by $a_1\vee a_2=(\lp(a_1)\vee \lp(a_2))b=(a_1a_1^\dag \vee a_2a_2^\dag )b$ and $a_1\wedge a_2=(\lp(a_1)\wedge \lp(a_2))b= (a_1a_1^\dag \wedge a_2a_2^\dag )b  $. 

\begin{theorem}\label{thm:supeinfgeneral} Let $\mathcal {A}$ be a regular Rickart *-ring ordered. For any $\stackrel{\alpha}{\preceq}$, if $a_1, a_2 \stackrel{\alpha}{\preceq} b$ then both $a_1\vee a_2$ and $a_1\wedge a_2$ exist in $\mathcal{A}$. Moreover,
\[a_1\vee a_2=b(\rp(a_1)\vee \rp(a_2))=b(a_1^\dag a_1\vee a_2^\dag a_2) \]
and
\[a_1\wedge a_2=b(\rp(a_1)\wedge \rp(a_2))= b(a_1^\dag a_1\wedge a_2^\dag a_2)  \]
whenever $\alpha\in \{ \crslc, \crsrc,\crsc, \cslc, \csrc,  \cscu \}$, while

\[a_1\vee a_2=(\lp(a_1)\vee \lp(a_2))b=(a_1a_1^\dag \vee a_2a_2^\dag )b \]
and
\[a_1\wedge a_2=(\lp(a_1)\wedge \lp(a_2))b= (a_1a_1^\dag \wedge a_2a_2^\dag )b  \]
whenever $\alpha\in \{\clslc, \clsrc, \clsc\}$.

\end{theorem}

\begin{proof} We first consider $\alpha\in \{ \cslc, \csrc,  \cscu \}$. Let $a_1, a_2, b\in \mathcal{A}$ be such that $a_i \stackrel{\alpha}{\preceq}b$, for $i\in \{1,2\}$. By Remark \ref{remarksupeinfenintervalos}, $a_i \stackrel{\alpha}{\preceq} b(\rp(a_1)\vee \rp(a_2))$. Assume that there exists $d$ such that $a_i\stackrel{\alpha}{\preceq} d \stackrel{\alpha}{\preceq} b(\rp(a_1)\vee \rp(a_2))$. In particular, $d \s b(\rp(a_1)\vee \rp(a_2))$. Since $b(\rp(a_1)\vee \rp(a_2))$ is the least upper bound of $\{a_1, a_2\}$ with respect to $\s$ by \cite[Theorem 4.2]{Ci2}, it follows that $d= b(\rp(a_1)\vee \rp(a_2))$. Therefore, $a_1\vee a_2= b(\rp(a_1)\vee \rp(a_2)) $ when $\alpha\in \{ \cslc, \csrc,  \cscu \}$. To study the infimum, Remark \ref{remarksupeinfenintervalos} yields $b(\rp(a_1)\wedge \rp(a_2))\stackrel{\alpha}{\preceq} a_i$. Suppose that there exists $c\in \mathcal{A}$ such that $ b(\rp(a_1)\wedge  \rp(a_2)) \stackrel{\alpha}{\preceq} c \stackrel{\alpha}{\preceq} a_i$. Then, $b(\rp(a_1)\wedge\rp(a_2))\s c$. Since $b(\rp(a_1)\wedge \rp(a_2))$ is the greatest lower bound of $\{a_1, a_2\}$ with respect to $\s$ by \cite[Theorem 4.2]{Ci2}, we obtain $c= b(\rp(a_1)\wedge \rp(a_2))$. Thus, $a_1\wedge a_2= b(\rp(a_1)\wedge \rp(a_2))$, when $\alpha\in \{ \cslc, \csrc,  \cscu \}$.

An analogous argument to the previous one proves the existence of the supremum and the infimum in  $\mathcal{A}$ when $\alpha\in \{\crslc, \crsrc, \crsc\}$. If $a_1$ and $a_2$ have a common upper bound $b$, then $b(\rp(a_1)\vee \rp(a_2))$ and $b(\rp(a_1)\wedge \rp(a_2))$ are respectively the supremum and infimum in the down-set of $b$ when $\alpha\in \{\crslc, \crsrc, \crsc\}$, by Remark \ref{remarksupeinfenintervalos}. The result follows by taking into account that these elements are also the supremum and infimum with respect to the right star partial order by \cite[Theorem 4.6]{CiarXiv}. 

We know that the involution $*$ induces the isomorphisms $\left(\mathcal{A}, \lslc\right)\simeq \left(\mathcal{A}, \rsrc\right)$, $\left(\mathcal{A}, \lsrc\right)\simeq \left(\mathcal{A}, \rslc\right)$ and $\left(\mathcal{A}, \lsc\right)\simeq \left(\mathcal{A}, \rsc\right)$. Hence, it is straightforward to see that $a_1\vee a_2=(\lp(a_1)\vee \lp(a_2))b$ and $a_1\wedge a_2=(\lp(a_1)\wedge \lp(a_2))b$ if $b$ is an upper bound of $a_i$, as a consequence of the existence of the supremum and infimum for $ \{\crslc, \crsrc, \crsc\}$.  
 \end{proof}

Recall that a Baer *-ring is a *-ring $\mathcal{A}$ such that the right annihilator of each subset of $\mathcal{A}$ is a principal right ideal generated by a self-adjoint idempotent. Moreover, $\mathcal{A}$ is a Baer *-ring if and only if $\mathcal{A}$ is a Rickart *-ring whose lattice $E(\mathcal{A})$ is complete \cite[Proposition 1.24]{Be}. 

The following generalizations of Theorem \ref{thm:latticesrs}, Corollary \ref{cor:sublatticesstar} and Theorem \ref{thm:latticesls} for regular Baer *-rings hold.

\begin{theorem}\label{thm:supInfBaerrsys} Let $\mathcal{A}$ be a \emph{regular} Baer *-ring ordered by any $\stackrel{\alpha}{\preceq}$ with $\alpha\in \{\crslc, \crsrc, \crsc, \cslc, \csrc, \cscu\}$. If $b$ is an upper bound of a nonempty subset $B\subseteq \mathcal{A}$ then:
\begin{enumerate}[(a)]
 \item $b\left( \bigvee \{\rp(c): c\in B\}\right)$ is the least upper bound of $B$ in $[0, b]^\alpha$. 
\item $b\left(\bigwedge \{\rp(c): c\in B\}\right)$ is the greatest lower bound of $B$ in $[0, b]^\alpha$.
 \end{enumerate} 
\end{theorem}
\begin{proof} Taking into account that $\bigvee \{\rp(c): c\in B\}$ and $\bigwedge \{\rp(c): c\in B\}$ exist, the result follows analogously to the proof of Theorem \ref{thm:latticesrs} and Corollary \ref{cor:sublatticesstar}. 
 \end{proof}

\begin{theorem}\label{thm:supremoRegularBaerDowntset} Let $\mathcal{A}$ be a \emph{regular} Baer *-ring ordered by any $\stackrel{\alpha}{\preceq}$ with $\alpha\in \{\ils, \clslc, \clsrc, \clsc\}$. If $b$ is an upper bound of a nonempty subset $B\subseteq \mathcal{A}$ then:
\begin{enumerate}[(a)]
 \item $\left(\bigvee \{\lp(c): c\in B\}\right) b$ is the least upper bound of $B$ in $[0, b]^\alpha$. 
\item $\left( \bigwedge \{\lp(c): c\in B\}\right)b$ is the greatest lower bound of $B$ in $[0, b]^\alpha$.
 \end{enumerate} 
\end{theorem}
\begin{proof} Since both $\bigvee \{\lp(c): c\in B\}$ and $\bigwedge \{\lp(c): c\in B\}$ exist, the result follows by an argument analogous to that used in the proof of Theorem \ref{thm:latticesls}. 
\end{proof}

A generalization of Theorem \ref{thm:supeinfgeneral} can be proved from Theorem \ref{thm:supremoRegularBaerDowntset} and Theorem \ref{thm:supInfBaerrsys}. 

\begin{theorem} Let $\mathcal {A}$ be a regular Baer *-ring ordered by any $\stackrel{\alpha}{\preceq}$. If $b$ is an upper bound of a nonempty subset $B\subseteq \mathcal{A}$ then the least upper bound of $B$ and the greatest lower bound of $B$ exist $\mathcal{A}$. Moreover,
\[\bigvee \{c: c\in B\}=b\left( \bigvee \{\rp(c): c\in B\}\right)\]
and
\[ \bigwedge \{c: c\in B\}=b\left( \bigwedge \{\rp(c): c\in B\}\right)\]
whenever $\alpha\in \{ \crslc, \crsrc,\crsc, \cslc, \csrc,  \cscu \}$, while

\[\bigvee \{c: c\in B\}=\left( \bigvee \{\lp(c): c\in B\}\right) b \]
and
\[ \bigwedge \{c: c\in B\}=\left( \bigwedge \{\lp(c): c\in B\}\right) b \]
whenever $\alpha\in \{ \ils, \clslc, \clsrc, \clsc\}$.
\end{theorem}

\section*{Funding} 
The authors were partially supported by projects PGI 24/L128 and PGI 24/ZL22, Departamento de Matemática, Universidad Nacional del Sur (UNS), Argentina.

\end{document}